\documentclass[11pt]{article}
\usepackage{srcltx}
\usepackage{eurosym}
\usepackage{mathtools}
\usepackage{amsmath}
\usepackage{amsfonts}
\usepackage{amssymb}
\usepackage{amsthm}
\usepackage{comment}
\usepackage{tensor} 

\allowdisplaybreaks

\usepackage{graphicx}
\usepackage{amsmath} 
\DeclareMathOperator{\sgn}{sgn}
\usepackage{mathrsfs}
\usepackage{xcolor}
\usepackage{exscale}
\usepackage{latexsym}
\usepackage{authblk}
\usepackage{booktabs}
\newtheorem*{assumption}{Assumption $\mathcal{A}$}
\usepackage[colorlinks,plainpages=true,pdfpagelabels,hypertexnames=true,colorlinks=true,pdfstartview=FitV,linkcolor=blue,citecolor=red,urlcolor=black]{hyperref}
\PassOptionsToPackage{unicode}{hyperref}
\PassOptionsToPackage{naturalnames}{hyperref}
\usepackage{enumerate}
\usepackage[shortlabels]{enumitem}
\usepackage{bookmark}
\usepackage{wasysym}
\usepackage[ddmmyyyy]{datetime}
\usepackage[margin=0.9in]{geometry}
\makeatletter
\g@addto@macro\normalsize{%
	\setlength\abovedisplayskip{4pt}
	\setlength\belowdisplayskip{4pt}
	\setlength\abovedisplayshortskip{4pt}
	\setlength\belowdisplayshortskip{4pt}
}
\numberwithin{equation}{section}
\everymath{\displaystyle}
\usepackage[capitalize,nameinlink]{cleveref}
\crefname{section}{Section}{Sections}
\crefname{subsection}{Subsection}{Subsections}
\crefname{condition}{Condition}{Conditions}
\crefname{hypothesis}{Hypothesis}{Conditions}
\crefname{lemma}{Lemma}{Lemmas}
\crefname{definition}{Definition}{Definitions}

\crefformat{equation}{\textup{#2(#1)#3}}
\crefrangeformat{equation}{\textup{#3(#1)#4--#5(#2)#6}}
\crefmultiformat{equation}{\textup{#2(#1)#3}}{ and \textup{#2(#1)#3}}
{, \textup{#2(#1)#3}}{, and \textup{#2(#1)#3}}
\crefrangemultiformat{equation}{\textup{#3(#1)#4--#5(#2)#6}}%
{ and \textup{#3(#1)#4--#5(#2)#6}}{, \textup{#3(#1)#4--#5(#2)#6}}%
{, and \textup{#3(#1)#4--#5(#2)#6}}

\Crefformat{equation}{#2Equation~\textup{(#1)}#3}
\Crefrangeformat{equation}{Equations~\textup{#3(#1)#4--#5(#2)#6}}
\Crefmultiformat{equation}{Equations~\textup{#2(#1)#3}}{ and \textup{#2(#1)#3}}
{, \textup{#2(#1)#3}}{, and \textup{#2(#1)#3}}
\Crefrangemultiformat{equation}{Equations~\textup{#3(#1)#4--#5(#2)#6}}%
{ and \textup{#3(#1)#4--#5(#2)#6}}{, \textup{#3(#1)#4--#5(#2)#6}}%
{, and \textup{#3(#1)#4--#5(#2)#6}}

\crefdefaultlabelformat{#2\textup{#1}#3}

\newtheorem{theorem} {Theorem}[section]
\newtheorem{proposition} [theorem]{Proposition}

\newtheorem{lemma}[theorem]{Lemma}

\newtheorem{counter example}[theorem]{Counter Example}
\newtheorem{remark}[theorem] {Remark}
\newtheorem{definition}[theorem] {Definition}

\def\CC{{\rm \kern.24em \vrule width.02em height1.4ex depth-.05ex \kern-.26emC}}

\def\TagOnRight

\def\AA{{it I} \hskip-3pt{\tt A}}

\def\QQ{\rlap {\raise 0.4ex \hbox{$\scriptscriptstyle |$}} {\hskip -0.1em Q}}

\makeatletter
\newcommand{\vo}{\vec{o}\@ifnextchar{^}{\,}{}}
\makeatother

\def\YYint#1#2#3{{\setbox0=\hbox{$#1{#2#3}{\iint}$}
		\vcenter{\hbox{$#2#3$}}\kern-.50\wd0}}

\def\XXint#1#2#3{{\setbox0=\hbox{$#1{#2#3}{\int}$}
		\vcenter{\hbox{$#2#3$}}\kern-.50\wd0}}

\makeatletter
\def\namedlabel#1#2{\begingroup
	\def\@currentlabel{#2}%
	\label{#1}\endgroup
}
\makeatother
\makeatletter
\newcommand{\rmh}[1]{\mathpalette{\raisem@th{#1}}}
\newcommand{\raisem@th}[3]{\hspace*{-1pt}\raisebox{#1}{$#2#3$}}
\makeatother

\newcounter{desccount}

\newcommand{\descref}[2]{\hyperref[#1]{\textnormal{\textcolor{black}{}\textcolor{blue}{ #2}\textcolor{black}{}}}}

\newcommand{\dref}[2]{\hyperref[#1]{\textcolor{black}{(}\textcolor{blue}{\bf #2}\textcolor{black}{)}}}
\newcommand{\be} {\begin{eqnarray}}
	\newcommand{\ee} {\end{eqnarray}}
\newcommand{\Bea} {\begin{eqnarray*}}
	\newcommand{\Eea} {\end{eqnarray*}}

\DeclareMathOperator{\lip}{{\rm Lip}}

\newcommand{\abs}[1]{\left| #1\right|}

\newcounter{whitney}
\refstepcounter{whitney}

\newcounter{ineqcounter}
\refstepcounter{ineqcounter}
\makeatletter
\def\ps@pprintTitle{%
	\let\@oddhead\@empty
	\let\@evenhead\@empty
	\def\@oddfoot{}%
	\let\@evenfoot\@oddfoot}
\makeatother
\usepackage[doublespacing]{setspace}
\usepackage[titletoc,toc,page]{appendix}

\makeatletter
\newcommand{\refcheckize}[1]{%
	\expandafter\let\csname @@\string#1\endcsname#1%
	\expandafter\DeclareRobustCommand\csname relax\string#1\endcsname[1]{%
		\csname @@\string#1\endcsname{##1}\wrtusdrf{##1}}%
	\expandafter\let\expandafter#1\csname relax\string#1\endcsname
}
\makeatother

\refcheckize{\cref}
\refcheckize{\Cref}


\makeatletter
\newcommand{\mainsectionstyle}{%
	\renewcommand{\@secnumfont}{\bfseries}
	\renewcommand\section{\@startsection{section}{2}%
		\z@{.5\linespacing\@plus.7\linespacing}{-.5em}%
		{\normalfont\bfseries}}%
}
\makeatother
\usepackage{pgf,tikz}
\usetikzlibrary{arrows}
\usetikzlibrary{decorations.pathreplacing}

\usepackage{xpatch}
\xpatchcmd{\MaketitleBox}{\hrule}{}{}{}
\xpatchcmd{\MaketitleBox}{\hrule}{}{}{}

\date{}

\usepackage{scalerel}

\makeatletter

\usepackage{subcaption}
\usepackage{hyperref}
\usepackage{caption}
\usepackage{subcaption}

\usepackage[utf8]{inputenc}
\allowdisplaybreaks
\usepackage{amsmath}
\usepackage{graphicx}
\usepackage{mathtools} 
\usepackage{hyperref}       
\usepackage{url}      
\usepackage{mathrsfs}
\usepackage{graphicx}
\usepackage{array}
\usepackage{color} 
\usepackage{tabularx}
\usepackage{amsmath,mathdots,amsthm}
\usepackage{amssymb}
\usepackage{amsfonts}
\usepackage{xcolor}
\usepackage{bm}
\usepackage{soul}
\usepackage{float}
\usepackage{cite}

\hypersetup{
	colorlinks=true,                          
	linkcolor=blue, 
	citecolor=blue, 
	urlcolor=black  
} 

\newtheorem{thm}{Theorem}[section]

\newtheorem{Lemma}[thm]{Lemma} 
\newtheorem{ass}[thm]{Assumption}

 \numberwithin{equation}{section}

\numberwithin{equation}{section}

\allowdisplaybreaks
\newcommand{\nrm}[1]{ \left\vert\left\vert #1 \right\vert\right\vert }

\newcommand{\pd}[2]{\frac{\partial #1}{\partial #2}}
\newcommand{\prn}[1]{\left( #1 \right)}

\newcommand{\I}{\mathbf{I}}
\newcommand{\Q}{\mathbf{Q}}
\newcommand{\hlf}{\frac{1}{2}}

\newcommand{\B}{\mathbf{B}}
\newcommand{\bS}{\mathbf{S}}
\newcommand{\Hess}[1]{\mathbf{H}_{#1}}

\newcommand{\rel}{^{\alpha}}
\newcommand{\intR}[1]{\int_\mathbb{R} #1 \, dx}
\newcommand{\eps}{\varepsilon}
\newcommand{\setR}{{\mathbb R}}

\newcommand{\dx}{\Delta x}
\newcommand{\dt}{\Delta t}
\newcommand{\F}{\mathbf{f}}

\newcommand{\red}[1]{\textcolor[rgb]{0.95,0.20,0.20}{#1}}
\newcommand{\pdt}[1]{\frac{\partial #1}{\partial t}}

\title{On hyperbolic approximations for a class of dispersive and diffusive-dispersive equations}
\author[1]{Rahul Barthwal\thanks{\href{mailto:rahul.barthwal@mathematik.uni-stuttgart.de}{rahul.barthwal@mathematik.uni-stuttgart.de}}}
\author[2]{Firas Dhaouadi\thanks{\href{mailto:firas.dhaouadi@bordeaux-inp.fr}{firas.dhaouadi@bordeaux-inp.fr}}}
\author[1]{Christian Rohde\thanks{\href{mailto:christian.rohde@mathematik.uni-stuttgart.de}{christian.rohde@mathematik.uni-stuttgart.de}}}

\affil[1]{\footnotesize Institute of Applied Analysis and Numerical Simulation, University of Stuttgart\\

Pfaffenwaldring 57, D-70569 Stuttgart, Germany}
\affil[2]{\footnotesize Institut de Mathématiques de Bordeaux, Université de Bordeaux, CNRS UMR 5251\\

Bordeaux INP, INRIA, F-33400, Talence, France}
\date{}

\hypersetup{
pdftitle={},
pdfauthor={},
pdfkeywords={},
}

\begin{document}
\maketitle
\begin{abstract}
We introduce novel approximate systems for dispersive and diffusive-dispersive equations with nonlinear fluxes. For purely dispersive equations, we construct a  first-order, strictly hyperbolic approximation.   
Local well-posedness of smooth solutions is achieved by constructing a unique symmetrizer that applies to arbitrary smooth fluxes. Under stronger conditions on the fluxes, we provide a strictly convex entropy for the hyperbolic system that corresponds to the energy of the underlying dispersive equation.
To approximate diffusive-dispersive equations, we rely on a viscoelastic damped system that is compatible with the found entropy for the hyperbolic approximation of the dispersive evolution. For the resulting hyperbolic-parabolic approximation, we provide a global well-posedness result. 
Using the relative entropy framework \cite{dafermos2005hyperbolic}, we prove that the solutions of the approximate systems converge to solutions of the original equations.\\
The structure of the new approximate systems allows to apply standard numerical simulation methods from the field of hyperbolic balance laws. We confirm the convergence of our approximations even beyond the validity range of our theoretical findings on set of test cases covering different target equations. We show the applicability of the approach for strong nonlinear effects leading to oscillating or shock-layer-forming behavior.

\end{abstract}
\medskip 
{\textbf{Mathematics subject classification.}  35A35, 35A01, 35G20, 35L45, 65M08.

\section{Introduction}
We consider for some interval $I \subset \setR$   scalar    balance laws of the form
\begin{equation}\label{main}
    u_t + {{f(u)}}_x = \varepsilon u_{xx}+  \gamma u_{xxx}~~{\rm{in}}~I_T:=I\times (0, T), \, ~T>0,
\end{equation}
with associated initial data 
\begin{align}\label{initialdata}
u(\cdot, 0)=u_0.
\end{align}
In \eqref{main}, $u=u(x,t)\in\mathbb{R}$ is the unknown. $f:\mathbb{R}\rightarrow \mathbb{R}$ denotes  a smooth flux function, $\gamma$ is a nonzero real parameter, $u_0:I\to \setR$ indicates the initial function.
We consider both cases, i.e.,  $\varepsilon=0$ leading to a purely dispersive evolution and the diffusive-dispersive equation with $\varepsilon>0$.\\
 Many mathematical models for nonlinear wave propagation in the low-frequency regimes lead to \eqref{main} with $\eps =0$. 
They have been analyzed extensively and we can only mention some examples: The model \eqref{main} represents the Korteweg-De Vries (KdV) equation for $f(u)=u^2/2$ \cite{whitham1965non}, the modified Korteweg De-Vries (mKdV) equation for $f(u)=u^3$ \cite{shearer2015traveling}, the generalized KdV equation for $f(u)=u^p+u, ~p\in \mathbb{N}$ \cite{bona2022numerical}, and the  Gardner equation for $f(u)=u^3+\mu u^2$ with $\mu\in \mathbb{R}$\cite{kamchatnov2012undular}.
The model \eqref{main} with $\eps >0$ can be considered as a toy model to describe phase transition dynamics when viscous and capillary effects are taken into account \cite{shearer2015traveling}. 
We refer the interested readers to the review paper\cite{el2017dispersive} for more details on dispersive and diffusive-dispersive equations and their physical significance; see also \cite{whitham2011linear, whitham1965non}.

Apart from its significance in physics, the model \eqref{main} comes with several mathematical challenges. In particular, detecting non-classical shock waves, such as nonlinear dispersive shocks or undercompressive shocks due to the coupling of the third-order term with the flux function, has gained attention, see, e.g., \cite{hayes1997non, hayes1999undercompressive, lefloch2004non}. In recent years, a lot of crucial work has been done for the analysis of diffusive-dispersive scalar balance laws of the form \eqref{main}. For the cubic-flux case $f(u)=u^3$ (the mKdV–Burgers model), the existence and structure of undercompressive traveling waves were rigorously characterized, with subsequent studies addressing stability; see, e.g., \cite{jacobs1995traveling, el2017dispersive, schulze1999undercompressive, shearer2015traveling} and references therein. Later, undercompressive dynamics in certain diffusive-dispersive systems were observed  numerically, offering strong confirmation and additional understanding of wave selection mechanisms \cite{bertozzi1998contact, Chalonsetal}. On the dispersive side, the broader theory of dispersive shock waves offers a unifying modulation framework for oscillatory shock-like structures generated by equations in the KdV/mKdV family and their viscous perturbations; see \cite{el2016dispersive, el2017dispersive}.\\
Notably, the equation \eqref{main}  comes with an energy-like structure. Sufficiently smooth solutions of \eqref{main} satisfy 
\begin{equation}\label{energylaw}
\frac{d}{dt} E [u] + \dfrac{d}{dx}Q[u]
=- \varepsilon \left(\sgn(\gamma) f'(u) u_x^2+|\gamma| u_{xx}^2\right),\,
\end{equation}
with  the  energy $E[u]$ and the energy flux $Q[u]$ given  by  
\begin{equation}\label{energy}
\begin{array}{rcl}
    E[u]&=& \sgn(\gamma) F(u)+\dfrac{|\gamma|  u_x^2}{2}, \quad F'(u)=f(u),\\[1.2ex]
    Q[u]&=& \dfrac{f(u)^2}{2}+|\gamma| (u_x f(u)_x-f(u)u_{xx})-\eps \sgn(\gamma) u_xf(u)-\eps |\gamma| u_x u_{xx}\\
    & &{} +\dfrac{\sgn(\gamma) \gamma^2 u_{xx}^2}{2}-\sgn(\gamma)\gamma^2 u_x u_{xxx}.
\end{array}
\end{equation}
If $\sgn(\gamma) F(u)$ is strictly convex or in other words $\sgn(\gamma)f'(u)\geq \delta$ for some $\delta>0$, then the energy $E$ becomes a strictly convex Lyapunov functional and \eqref{energylaw} becomes an energy inequality, which then provides some $H^1$-control on $u$. In particular, for the dispersive case ($\eps=0$), the energy $E$ is conserved. Relying on this energy structure and higher-order energy estimates, classical results for (generalized) KdV/mKdV establish local and global well-posedness in $ {H}^{s}$-spaces, which provide the analytical foundation for studying solutions emanating from sufficiently regular data; see, e.g., \cite{kenig1991well, tao2006nonlinear, killip2019kdv, Colliander2003, farahkdv2011} and references cited therein. \\
Beyond direct analysis of \eqref{main}, there have been attempts to provide lower-order approximations for \eqref{main}. Such relaxations are typically tailored to specific examples of \eqref{main}  so as to preserve key structural features (e.g., energy preservation or dissipation) while simplifying the dynamics. For the diffusive–dispersive law \eqref{main} with $\gamma>0$, Corli\&Rohde \cite{corli2012singular} proposed an elliptic relaxation in which an auxiliary variable is introduced and coupled to a reduced-order evolution through a screened Poisson equation. Their approximation takes the form.
\begin{equation}
    u\rel_{t} + f(u\rel)_x = \varepsilon\, u\rel_{xx} - \alpha \prn{u\rel - c\rel}_x, 
    \quad -\gamma  c\rel_{xx} = \alpha \prn{u\rel - c\rel}.
    \label{eq:scrn_poiss}
\end{equation}
Relying on the energy structure of the approximate system, it was proven  that as $\alpha\rightarrow \infty$, the sequence of solutions ${\{(u\rel, c\rel)\}}_{\alpha>0}$ of the initial value problem for \eqref{eq:scrn_poiss}
converge to some function $(u, c)$, such that $u$ solves \eqref{main}, \eqref{initialdata}.
Extending this idea, Corli et al.~\cite{corli2014parabolic} introduced a parabolic relaxation by adding temporal dynamics for the auxiliary variable. The parabolic relaxation system for $\gamma>0$ takes the form.
\begin{equation}
    u\rel_{t} + f(u\rel)_x = \varepsilon\, u\rel_{xx} - \alpha \prn{u\rel - c\rel}_x, \quad 
    \beta c\rel_t-\gamma \,c\rel_{xx} = \alpha \prn{u\rel - c\rel},
    \label{eq:relaxed_scrn_heat}
\end{equation}
where the parameter $\beta$ is related with the relaxation parameter $\alpha$ such that $\beta(\alpha)\rightarrow 0$ as $\alpha\rightarrow \infty$. By higher-order energy estimates,  convergence of the sequence of solutions ${\{(u\rel, c\rel)\}}_{\alpha>0}$ of \eqref{eq:relaxed_scrn_heat} to the solutions of \eqref{main}, \eqref{initialdata} has been established as $\alpha \rightarrow \infty$. While both elliptic and parabolic approximations capture the diffusive–dispersive behavior effectively, they still result in 
second-order and mixed-type systems. These 
cannot be treated by standard numerical methods. In particular, the numerical simulation becomes demanding in the stiff limit 
regime for $\alpha \to \infty$.
On the other hand,  first-order hyperbolic or second-order hyperbolic-parabolic systems admit a broad class of high-resolution methods (like, e.g., finite volume methods with classical Riemann solvers and limiters), making them straightforward to implement and comparatively inexpensive; see, e.g.,~\cite{dhaouadi2022first, dhaouadi2025first}. 
\medskip

In this article, therefore, we propose alternative hyperbolic and hyperbolic–parabolic systems that approximate solutions of 
\eqref{main}, \eqref{initialdata} while retaining essential structural features and enabling efficient computation. The proposed approximations
also rely on the introduction of an auxiliary variable, like in \eqref{eq:scrn_poiss}, \eqref{eq:relaxed_scrn_heat}, but let it evolve by special wave equations. However, our proposed system does not rely on the sign of dispersion in contrast to the elliptic and parabolic relaxation systems and considers both the cases $\gamma>0$ and $\gamma<0$.\\
In Section \ref{sec: model}, we first formulate the suggested approximation models for the purely dispersive case $\eps =0$ in \eqref{main}.  By transformation of the unknowns, we can recast 
the approximation in the form of a first-order $(4\times 4)$-system of balance laws; see \eqref{hyperbolic_system}}. This system turns out to be strictly hyperbolic, admits explicit Riemann invariants, and preserves the Hamiltonian structure of \eqref{main}. For monotone fluxes, we can provide a convex entropy which resembles the energy structure of \eqref{main} as in \eqref{energylaw}. 
In the second part of Section \ref{sec: model}, we provide a hyperbolic-parabolic approximation in \eqref{hyperbolic_parabolic_system} using a damped wave operator. 
Besides local well-posedness, we establish a global well-posedness result
for initial data close to constant equilibrium in Theorem \ref{global_wellposedness}. Again, we exploit the entropy structure which is inherited from the hyperbolic approximation \eqref{hyperbolic_system} and ensures dissipative behavior.  
The main novel contributions of Section \ref{sec: model} are the two hyperbolic and hyperbolic-parabolic approximations \eqref{hyperbolic_system} and \eqref{hyperbolic_parabolic_system}.\\
Section \ref{sec: convergence} is devoted to prove the convergence of the solutions of the proposed systems to solutions of \eqref{main} using the relative entropy framework. As our main analytical contributions, we prove in  Theorem \ref{relative_entropy_theorem_1}  the convergence towards solutions of \eqref{main} in the purely dispersive case and verify in Theorem \ref{theo:dd}  the convergence in the diffusive-dispersive case. Let us note that we  consider in Theorem \ref{relative_entropy_theorem_1} weak entropy solutions of the 
approximate system \eqref{hyperbolic_system}, whereas Theorem \ref{theo:dd} concerns classical solutions of \eqref{hyperbolic_parabolic_system} which exist  globally under the assumptions of Theorem \ref{global_wellposedness}.\\
In Section \ref{sec: Numerics}, we provide several test cases for a variety of dispersive and diffusive-dispersive equations. We show that 
the approximations \eqref{hyperbolic_system} and \eqref{hyperbolic_parabolic_system} can be solved numerically by a standard approach from the realm of hyperbolic balance laws. The method applies to all test cases, including those with strong nonlinear effects. The study confirms the theoretical convergence results from Section \ref{sec: convergence}.\\
Conclusions and future outlooks are provided in Section \ref{sec: conclusions}.
\medskip
%

%
%

We conclude the introduction with some remarks on the available literature on first-order approximations for diffusive-dispersive equations. Such models have long been studied in diverse settings. In particular, many hyperbolic relaxation models for higher-order partial differential equations have been suggested in recent years, see, e.g., \cite{biswas2025traveling, dhaouadi2025first, favrie2017rapid, gavrilyuk2022hyperbolic, gavrilyuk2024conduit,ranocha2025high} and references cited therein. Numerical investigations in all these works suggest a good convergence towards the solutions of the original equations, but mostly lack a rigorous convergence analysis or do not exploit or rely on the energy structure of the original problem. In particular, a rigorous justification of the Favrie-Gavrilyuk 
approximation \cite{favrie2017rapid} to the Serre-Green-Naghdi model was provided in \cite{duchene2019rigorous}; however, most of the other relaxation systems still lack rigorous convergence proofs. A more recent contribution by Giesselmann\&Ranocha \cite{giesselmann2025convergence} addresses this gap by proving the convergence for several hyperbolic relaxation systems via the relative entropy framework. Building on this idea, we show that the same approach applies to our proposed approximate system \eqref{hyperbolic_system} using its convex entropy. 
We refer the interested readers to \cite{giesselmann2017relative, dafermos2005hyperbolic, diperna1979uniqueness} and references cited therein for more details on the relative entropy framework. 
\subsection*{Notations}\label{Notation}
Throughout the article, we use the following standard notations. We denote by $C^k(D)$, the space of $k$-times continuously differentiable functions defined in some domain $D$ and $k\in \mathbb{N} \cup \{ \infty \}$. By $C^{\lip}_0(D)$, we denote the space of Lipschitz continuous functions having a compact support. 
Further, for $p \in [1,\infty) \cup \{\infty\}$, we denote the Lebesgue space on $D$ as $L^p(D)$ and the $L^p(D)$-norm as ${\lVert \cdot \rVert}_{L^p(D)}$, which for a function $g \in L^p(D)$ is defined as
\[
    {\lVert g\rVert}_{L^p(D)} = \biggl(\int_{D} |g(x)|^p \,dx\biggr)^{\frac{1}{p}} ~ \text{for}~p \in[1,\infty), \quad 
{\lVert g \rVert}_{L^{\infty}(D)}
    = {\rm esssup}_{x\in D} \big\{ |g| \big\}  ~\text{for}~ p =\infty.\]
We denote by $H^k(D)$ the Sobolev space of order $k$ consisting of 
square-integrable functions whose weak derivatives up to order $k$ are also square-integrable, that is
\[
    H^k(D)
    := \bigl\{ g \in {L}^{2}(D) : \partial^m g \in {L}^{2}(D)
        \text{ for all multi-indices}~ m \text{ with } |m| \le k \bigr\},
\]
equipped with the norm
\[
    {\lVert g \rVert}_{H^k(D)}^2
    := \sum_{|m| \le k} {\lVert \partial^m g \rVert}_{{L}^{2}(D)}^2,
\]
where $\partial^m g$ denotes the weak derivative of $g$ of order $m$. \\
Finally, for a Banach space $X$ and $T>0$, we denote by $L^p(0, T;X)$, the Bochner space of measurable functions $g:(0,T)\to X$ with an induced norm
\[
    {\lVert g \rVert}_{L^p(0,T;X)}
    = \biggl( \int_0^T {\lVert u(t) \rVert}_{X}^p \,\mathrm{d}t \biggr)^{1/p},
\]
for $p \in [1,\infty)$, while for $p = \infty$
\[
{\lVert g \rVert}_{L^{\infty}(0,T;X)}
    = \sup_{t\in[0, T]} {\lVert g(t) \rVert}_{X}.
\]

\section{Hyperbolic and hyperbolic-parabolic approximations\label{sec: model}}
In this section, we propose the relaxation approximations for both the purely  dispersive equation \eqref{main} with $\eps=0$ in Section \ref{sec:main1} 
and the diffusive-dispersive equation \eqref{main} with $\eps >0$ in Section \ref{sec:main2}.
The approximations can be seen as extensions of the elliptic approach from 
\cite{corli2012singular}, but replacing the elliptic operator by different types of second-order hyperbolic wave operators. These will be reformulated
by introducing first- and second-order systems, which are shown to be of hyperbolic and hyperbolic-parabolic structure, respectively. In the dispersive case, we show that the proposed system is strictly hyperbolic and is locally well-posed via a unique symmetrizer. Moreover, we succeed in showing that the first-order system is derivable from Hamilton's principle and conserves an associated entropy. The approach is then extended to the dissipative case, for which we succeed in providing a complete entropy framework (see \cite{dafermos2005hyperbolic}) that is compatible with the energy related to the limit problem \eqref{main}, \eqref{initialdata}. We exploit the entropy structure and prove local and global well-posedness results for the corresponding initial value problems.

\subsection{A first-order approximation of the purely dispersive equation\label{sec:main1}}

We start from \eqref{main} with $\eps =0$.  Thus, we consider for $\gamma \in \setR\setminus \{0\}$, the dispersive equation
\begin{equation}
	u_t + 
    {f(u)}_x = \gamma u_{xxx},~~\text{in}~I_T,
	\label{eq:dispersive}
\end{equation}
subject to periodic boundary conditions and the 
initial condition
\begin{equation}
	u(\cdot,0) = u_0 \mbox{ in } I.
	\label{eq:dispersive_ini}
\end{equation}
The existence of classical solutions for such dispersive equations with periodic boundary conditions has been analyzed in various settings. In particular, for the generalized KdV equation, local and global well-posedness for high regularity solutions in $H^s$-spaces is a well-established result. We refer the interested readers to the classical work of Kenig et al. \cite{kenig1991well} in this regard, where the local well-posedness of $H^s$ solutions for $s>3/4$ was developed. For global well-posedness of smooth solutions in $H^s$-spaces for modified and generalized KdV equations in both $\mathbb{R}$ and $\mathbb{T}$ (periodic BCs), we refer the interested readers to Colliander et al. \cite{Colliander2003}, Farah et al. \cite{farahkdv2011}, and references therein. For a more general flux function and variable dispersion, we refer to the recent work of Mietka \cite{mietka2017well}. In particular, for results related to the periodic setting, we also refer to \cite{staffilani1997solutions, guan2014kdv}.\medskip \\
As outlined in the introduction, an important property of 
classical solutions of \eqref{eq:dispersive} is the energy conservation such that for the energy/energy flux pair $(E,Q)$ from  \eqref{energy}, we have the relation
\begin{equation}\label{gradientflow}
\dfrac{d}{dt}{E[u]}+ \dfrac{d}{dx}{Q[u]}=0.
\end{equation}
We search now for a first-order approximation of \eqref{eq:dispersive} that resembles
this gradient flow structure. 
Analogously to the screened Poisson equation \eqref{eq:scrn_poiss} and the screened heat equation \eqref{eq:relaxed_scrn_heat}, we relax the Laplacian term $u_{xx}$ in \eqref{eq:dispersive} by a screened linear wave operator. 
For some relaxation parameter $\alpha >0$,  we  seek then for a pair 
$(u\rel,c\rel)  $  that satisfies  in $I_T$ the approximate system   
\begin{equation}\label{eq:waveform}
	\begin{array}{rcl} 
		u\rel_t + f(u\rel)_x &= & - \sgn(\gamma)\alpha \prn{u\rel - c\rel}_x, \\[1.1ex]
		\beta c\rel_{tt}-|\gamma| c\rel_{xx} &= & 
        \alpha \prn{u\rel - c\rel},
	\end{array}
\end{equation}	
subject to  periodic boundary conditions, and the prepared initial data 
\begin{equation}\label{eq:wave_equation_ini}  
u\rel (\cdot,0) = c\rel(\cdot,0) =  u_0, \quad c\rel_t(\cdot,0) =    - {f(u_0)}_x - |\gamma| u_{0,xxx}, \quad \alpha(u\rel-c\rel) (\cdot,0)= - |\gamma| u_{0,xx} \quad \mbox{ in } I. 
\end{equation}
The parameter $\beta$   in ${\eqref{eq:waveform}}_2$ will be expressed in the 
sequel in terms of $\alpha$.  In the limit $\alpha\rightarrow \infty$ we expect to recover \eqref{eq:dispersive} if $\beta=\beta(\alpha)$ vanishes. Naturally, the characteristic speed associated with the wave operator ${\eqref{eq:waveform}}_2$ given by $v_\beta := \sqrt{|\gamma|/\beta}$ goes to $\infty$ as $\beta \to 0$, in line with the instantaneous signal speed of the elliptic operator in \eqref{eq:scrn_poiss}.\medskip \\
Now, in order to reduce the coupled mixed-order system \eqref{eq:waveform} to an equivalent purely first-order form, let us consider a change of variables following \cite{dhaouadi2019extended,dhaouadi2025first}, i.e.,
\begin{equation}
		w\rel 
        = c\rel_t, \, \quad
		p\rel =   c\rel_x, \quad 
		\psi\rel = \alpha \prn{u\rel - c\rel}.  
	\label{eq:w_p_psi}
\end{equation}
We then substitute these quantities  into \eqref{eq:waveform} and supply the expressions for the time-derivatives of $p\rel$ and $\psi\rel$  as closure equations to obtain a system of first-order
balance laws of the form
\begin{equation}
\begin{aligned}
	u\rel_t + \prn{f(u\rel) + \sgn(\gamma) \psi\rel}_x &=\phantom{-}0, \\
	\psi\rel_t+\alpha \prn{ f(u\rel)+\sgn(\gamma) \psi\rel}_x &= -\alpha\, w\rel,\\
		w\rel_{t}-v_\beta^2\, p\rel_{x} &= \phantom{-}\beta^{-1}\,  \psi\rel, \\
	p\rel_t - \phantom{v_\beta^2}w\rel_x&=\phantom{-}0
\end{aligned}  \qquad \mbox{ in } I_T.
	\label{hyperbolic_system}
\end{equation}
Note that if $w{\rel}\rightarrow u_t$, $\psi{\rel}\rightarrow -|\gamma| u_{xx}$ and $p{\rel}\rightarrow u_x$ as $\alpha\rightarrow \infty$, we recover the original equation \eqref{eq:dispersive}. Thus, the system \eqref{hyperbolic_system} appears to be a consistent approximation for \eqref{eq:dispersive}. 
\medskip \\
To maintain the consistency with the initial data from \eqref{eq:wave_equation_ini}, the  prepared  initial conditions 
\begin{equation}
u\rel(\cdot,0)= u_0,   \,  \psi\rel(\cdot,0)= -|\gamma| u_{0,xx},\,
w\rel(\cdot,0) =  - {f(u_0)}_x - |\gamma| u_{0,xxx}, \, 
p\rel(\cdot,0) =   u_{0,x} \mbox{ in }I
    \label{IC_BC_Hyp}
\end{equation}
are supposed to hold.\\
In Theorem \ref{relative_entropy_theorem_1} we prove that solutions of the 
periodic initial value problem \eqref{hyperbolic_system}, \eqref{IC_BC_Hyp}
converge for $\alpha \to \infty$  to solutions of \eqref{eq:dispersive}, \eqref{eq:dispersive_ini} if $\beta = \alpha^{-1}$ is chosen.
\subsection{Qualitative properties of the first-order approximation}
\subsubsection{Hyperbolicity, symmetrizability and local well-posedness}\label{sec: hyperbolicity}
In this section we omit the superscript $(.)\rel$ on the approximation unknowns  and simply write ${\mathbf U} = (u, \psi, w, p)^T$ for  $ {\mathbf U}\rel = (u\rel, \psi\rel, w\rel, p\rel)^T\in  \setR^4$, if no confusion is caused.\\ 
In order to verify the hyperbolicity of the system \eqref{hyperbolic_system}, we  consider a smooth solution $\mathbf U$ and  cast \eqref{hyperbolic_system} into the balance law  form
\begin{equation}
\mathbf{U}_t + {{\mathbf f}(\mathbf{U})}_x = \mathbf{b}(\mathbf{U}) \Longleftrightarrow
	\mathbf{U}_t + \mathbf{Df}(\mathbf{U})\mathbf{U}_x = \mathbf{b}(\mathbf{U}),
    \label{eq:homogeneous_system}
\end{equation} 
where $\mathbf{Df}(\mathbf{U})$ is the Jacobian of the flux ${\mathbf f} = {\mathbf f}({\mathbf U})$. The flux   and the linear source ${\mathbf b} = {\mathbf b}( {\mathbf U})$ are given by 
\begin{equation*}
\mathbf{f}(\mathbf{U}) = 
	\left(
	\begin{array}{c}
		f(u)+ \sgn(\gamma)\psi \\
		\alpha (f(u)+\sgn(\gamma)\psi) \\
	-	v^2_\beta p\\
	 -w 
	\end{array}
	\right), \quad 
\ \mathbf{b}(\mathbf{U}) = \mathbf{S} \mathbf{U},
\, \,\quad
\mathbf{S} = \begin{pmatrix}
0 &0 &0 & 0 \\
0 &0 &-\alpha & 0 \\
0 &\beta^{-1} &0 & 0 \\
0 &0 &0 & 0 
\end{pmatrix}, 
%
%
\end{equation*}
such that  the Jacobian of the flux is explicitly given by
\begin{equation*}
 \mathbf{Df}(\mathbf{U}) = 
\left(
\begin{array}{cccc}
	f'(u) & \sgn(\gamma)   & 0 & 0 \\
	\alpha f'(u) & \alpha \sgn(\gamma)  & 0 & 0 \\
	0 & 0 & 0 & -v_\beta^2 \\
	0 & 0 & -1 & 0 \\
\end{array}
\right).
\end{equation*}
The eigenvalues of $\mathbf{Df}(\mathbf{U})$  compute as the    real numbers $\lambda_1, \ldots, \lambda_4$ given by 
\begin{equation}
	\lambda_1=- v_\beta,\quad \lambda_2=0, \quad \lambda_{3} = v_\beta, \quad \lambda_4 = f'(u)+\alpha \sgn(\gamma).
\end{equation}
While the ordering $\lambda_1 < \lambda_2 < \lambda_3$ is direct, the placement of $\lambda_4$ within this ranking depends on the scaling of the free parameters $\alpha$ and $\beta$. In all what follows, we choose $\alpha$ such that $\abs{\lambda_4} > \lambda_3$ and ensure $\lambda_4$ does not change sign under the sufficient condition  
\begin{gather}
	\alpha > \max \abs{f'(u)}.
\end{gather}
Under these conditions, depending on the dispersion sign one obtains 
\begin{align*}
	\lambda_1 < \lambda_2 < \lambda_3 < \lambda_4 \quad \text{for} \ \sgn(\gamma) = 1, \qquad  \lambda_4<\lambda_1 < \lambda_2 < \lambda_3 \quad \text{for} \  \sgn(\gamma) = -1.
\end{align*}
In both cases the approximate system \eqref{hyperbolic_system} is  strictly hyperbolic. The corresponding right eigenvectors  ${{\bf r}_1,\ldots,{\bf r}_4} \in \setR^4$ can be explicitly computed and are given by 
\begin{equation}\label{eigen-vectors}
	{\mathbf{R}} :=  \big({\bf r}_1|\cdots|{\bf r}_4\big) =  \left(
	\begin{array}{cccc}
		0 & - \sgn(\gamma) & 0 & 1 \\
		0 & f'(u) & 0 & \alpha \\
		v_{\beta} & 0 & -v_{\beta} & 0 \\
		1 & 0 & 1 & 0
	\end{array}
	\right).
\end{equation} 
We proceed to prove that the hyperbolic system \eqref{hyperbolic_system} can be symmetrized using a unique symmetrizer. To be precise, we prove the following.
\begin{Lemma}[Symmetrizability and local well-posedness of the initial value problem for \eqref{hyperbolic_system}]\label{lem:wp}
For $\alpha > \max \abs{f'(u)}$, the hyperbolic system \eqref{hyperbolic_system} can be converted into an equivalent symmetric hyperbolic form. In particular, the hyperbolic system \eqref{hyperbolic_system} is a Friedrichs-symmetrizable system.\\
For
\[
\mathbf{U}_0 := (u_0, \psi_0, w_0, p_0)^{\top} \in ({H}^{s}(I))^4,~s > 3/2,
\]
there exists a time $T^*\in (0, \infty)$ such that the initial value problem  for \eqref{hyperbolic_system}  and     $\mathbf{U}(\cdot,0)= \mathbf{U}_0 $ in $I$
is solved in $I\times [0, T^*)$ by a unique function  $ \mathbf{U}$
such that
\[
\mathbf{U}(\cdot, t) \in C\!\left([0, T^*); ({H}^{s}(I))^4\right) 
\cap C^1\!\left([0, T^*); (H^{s-1}(I))^4\right)
\]
holds for $ s>3/2$. 
\end{Lemma}
\begin{proof}
In order to prove this lemma, first consider the  map
\begin{equation*}
	\mathcal{T}: \mathbb{R}^4 \mapsto \mathbb{R}^4, \qquad 
	\mathbf{U} = \left(
	\begin{array}{c}
		u \\
		\psi \\
		w \\
		p 
	\end{array}
	\right) \longmapsto 
	\boldsymbol{\Phi} = 	\left(
	\begin{array}{c}
		\phi_1 \\
		\phi_2 \\
		\phi_3 \\
		\phi_4 
	\end{array}
	\right), \quad \text{with}~ 	\begin{cases}
		\phi_1 =   \sgn(\gamma) f(u) + \psi, \\
		\phi_2 = \alpha \,u - \psi, \\
		\phi_3 = v_\beta\,p, \\
		\phi_4 = w.
	\end{cases} 
\end{equation*}
The Jacobian of this transformation is given by
\begin{equation*}
	D\mathcal T(\mathbf{U}) = \left(
	\begin{array}{cccc}
		 \sgn(\gamma) f'(u) & 1  & 0 & 0 \\
		\alpha   & \!\!\!\!-1  & 0 & 0 \\
		0 & 0 & 0 & v_\beta \\
		0 & 0 & 1 & 0 \\
	\end{array}
	\right), 
\end{equation*}
{such that}  $\det\prn{D\mathcal T(\mathbf{U})} = v_\beta\lambda_4\ne0$.
This shows that the mapping $\mathcal{T}$ is bijective, and thus the system \eqref{eq:homogeneous_system} can be converted into the equivalent form
\begin{equation*}
		\boldsymbol{\Phi}_t + \mathbf{\tilde A}(\boldsymbol{\Phi})\boldsymbol{\Phi}_x = \mathbf{\tilde b}(\boldsymbol{\Phi}),
\end{equation*} 
with the symmetric matrix $\mathbf{\tilde A}= \mathbf{\tilde A}(\boldsymbol{\Phi})$  given by
\begin{equation*}
\mathbf{\tilde A}(\boldsymbol{\Phi}) = \left(
	\begin{array}{cccc}
		  \sgn(\gamma)\alpha + f'(u(\boldsymbol{\Phi})) & 0  & 0 & 0 \\
		0   & 0  & 0 & 0 \\
		0 & 0 & 0 & -v_\beta \\
		0 & 0 & -v_\beta &  0 \\
	\end{array}
	\right)
\end{equation*}
and $\mathbf{\tilde b}(\boldsymbol{\Phi})=(0, -\alpha \phi_4, \psi(\phi_1, \phi_2), 0)^\top$ such that $\psi(\boldsymbol{\Phi})$ satisfies $\sgn(\gamma) f(u(\boldsymbol{\Phi}))+\psi=\phi_1$ and $\alpha u(\boldsymbol{\Phi})-\psi=\phi_2$. Here, the quantity $u(\boldsymbol{\Phi})$ is simply the value of $u$, computed from $\boldsymbol{\Phi}$ through the inverse transform $\mathcal{T}^{-1}$.\\
Clearly, under this change of variables, the system \eqref{hyperbolic_system} is a Friedrichs-symmetrizable system, and thus the local well-posedness of the Cauchy problem for classical solutions is a direct consequence of classical results on the theory of hyperbolic balance laws, see, e.g.,~\cite{dafermos2005hyperbolic}.
\end{proof}
\begin{remark} In Section \ref{sec:entropy}
below we show that \eqref{hyperbolic_system} is equipped with an 
entropy/entropy-flux pair with a convex entropy provided that  $\sgn(\gamma)f'(u)\geq \delta>0$ holds. This observation
also leads to a local well-posedness result like that in Lemma \ref{lem:wp}.
But note that  Lemma \ref{lem:wp} requires $f$ to be smooth only, without any  assumption on the sign of $f'$.  
\end{remark}
\subsubsection{Riemann invariants}
We explore the mathematical structure of the system \eqref{hyperbolic_system} and prove in this section that the system \eqref{hyperbolic_system} possesses an entire set of linearly independent Riemann invariants, which allows us to convert the system \eqref{hyperbolic_system} into an equivalent diagonal form as well. 
\begin{Lemma}[Riemann invariants for the system \eqref{hyperbolic_system}]
The	system \eqref{hyperbolic_system} possesses a full set of Riemann invariants, which form a coordinate system. 
The Riemann invariants $R_1=R_1({\bf{U}}), \ldots, R_4 = R_4(\bf{U})$ are given by  
	\begin{equation*}
		R_1 ({\bf{U}})= w +v_\beta\, p , \quad R_2({\bf{U}}) = \alpha\, u - \psi, \quad R_3 ({\bf{U}})= w - v_\beta \, p, \quad R_4({\bf{U}}) =  \sgn(\gamma) f(u) +  \psi. 
	\end{equation*}
\end{Lemma}
\begin{proof}
It suffices to show that the row vectors $\boldsymbol{\ell}_i$ defined as
\begin{equation*}
\boldsymbol{\ell}_i=	\boldsymbol{\ell}_i(\mathbf{U})=  \big({R_{i,u}}(\mathbf{U}),  {R_{i,\psi}}(\mathbf{U}), {R_{i,w}}(\mathbf{U}), {R_{i,p}}(\mathbf{U})  \big) \quad    (i=  1,\ldots, 4, \, \mathbf{U} \in \setR^4),
\end{equation*}
are each a left eigenvector associated with $\lambda_i$. A direct computation gives
\begin{equation*}
	\boldsymbol{\ell}_1 = \prn{0,0, 1, v_\beta},\quad  \boldsymbol{\ell}_2 = \prn{\alpha, -1, 0, 0}, \quad \boldsymbol{\ell}_3 = \prn{0,0,1,-v_\beta}, \quad \boldsymbol{\ell}_4 = \prn{ \sgn(\gamma) f'(u),1,0,0},
\end{equation*}     
and straightforward algebra shows that 
\begin{equation*}
	\boldsymbol{\ell}_i\,\mathbf{Df}(\mathbf{U})~=~\lambda_i\, \boldsymbol{\ell}_i\quad (i=  1,\ldots, 4),
\end{equation*} 
which establishes the result.
\end{proof}
Note that the full set of explicit Riemann invariants allows us to make the system completely diagonal along the characteristic directions. This allows, e.g., to find exact solutions of the  Riemann problem  for the  hyperbolic system \eqref{hyperbolic_system} or to devise proper 
boundary conditions  \eqref{hyperbolic_system} beyond our periodic choice.
\subsubsection{Entropy/entropy-flux pair and weak entropy solutions of \eqref{hyperbolic_system}}\label{sec:entropy}
In Section \ref{sec: hyperbolicity}, we have already proven that the system \eqref{hyperbolic_system} is a Friedrichs-symmetrizable system and that the associated initial value problem is locally well-posed for classical solutions.
Since \eqref{hyperbolic_system} is a system of nonlinear hyperbolic balance laws, we also require a notion of weak solvability. In fact,  we show that \eqref{hyperbolic_system}  is equipped with an entropy/entropy-flux pair if the flux $f$ is monotone.  Then we can 
naturally define weak entropy solutions, see, e.g.,~\cite{dafermos2005hyperbolic}.\\
Let us first collect some assumptions on the behavior of $f$ and some growth conditions needed later on in Section \ref{sec: convergence}. 
\begin{assumption}\namedlabel{ass_fmonotone}{\ensuremath{{\mathcal{A}}}}
Let $L>0$ be a constant such that the smooth function  $f$ from \eqref{main}
satisfies 
\[
|f'(u)|,~| f''(u)|~ \le L \quad (u\in \setR).
\]
Further,  there is a constant $\delta >0$ such that the lower bound $\sgn(\gamma)f'(u)\geq \delta >0$ holds for all $u\in \setR$.  
\end{assumption}%
Note that the latter condition ensures that the function
\[ \sgn(\gamma)F(u) = \sgn(\gamma)\int_0^u
 f(v)\, dv
 \] is uniformly convex in $\setR$. It is satisfied for the choice  $|f(u)|=au^{2n-1}+bu, ~(n\in \mathbb{N})$ with $a,b >0$ which has been used in, e.g., the seminal works \cite{bona2022numerical,benjamin1972model}.
 
Under the Assumption \ref{ass_fmonotone}, we can provide an entropy/entropy-flux pair for the system \eqref{hyperbolic_system}. By an entropy/entropy-flux pair for the $\alpha$-parameterized system \eqref{hyperbolic_system}, we mean a pair 
$(E^{\alpha}, Q^{\alpha}): \mathbb{R}^4\mapsto \mathbb{R}^2$ such that $E^{\alpha}$ is convex  with positive definite  Hessian $\nabla^2_{{\mathbf U}} E^{\alpha}$ in $\setR^4$ and such that the pair  $(E^{\alpha}, Q^{\alpha})$ satisfies 
\begin{equation}\label{entropy}
\nabla_{\mathbf{U}} E^{\alpha}(\mathbf{U})^\top\mathbf{Df(U)}=\nabla_{\mathbf{U}}Q^{\alpha}(\mathbf{U})^\top
 \text{ for all $\mathbf{U}\in \mathbb{R}^4$}.
\end{equation}
A specific entropy/entropy-flux pair for the system \eqref{hyperbolic_system} is given in the following lemma.
\begin{Lemma}[Entropy/entropy-flux pair for \eqref{hyperbolic_system}] \label{lemma_entropy}
The pair $(E^{\alpha}, Q^{\alpha})$ with 
\begin{equation}\label{convex-entropy}
    E^{\alpha}({\mathbf U}) = \sgn(\gamma) f(u) + \frac{|\gamma|}{2} {p}^2 + \frac{1}{2\alpha} \psi^2+ \frac{\beta }{2}{w}^2, 
\end{equation}
and 
\begin{align}\label{entropy-flux}
    Q^{\alpha}({\mathbf U})= \dfrac{1}{2}f(u)^2+f(u) \psi+\sgn(\gamma)\dfrac{\psi^2}{2}- |\gamma| pw 
\end{align}    
is an entropy/entropy-flux pair for \eqref{hyperbolic_system}. 
\end{Lemma}
\begin{proof} 
In order to prove this lemma, we first compute the gradients of $E^{\alpha}$ and $Q^{\alpha}$ as
\begin{align*}
    \nabla_{\mathbf{U}} E^{\alpha}(\mathbf{U})&=(\sgn(\gamma) f(u),  \psi/\alpha, \beta w, |\gamma| p)^\top,\\
    \nabla_{\mathbf{U}} Q^{\alpha}(\mathbf{U})&=(f'(u)(f(u) +\psi), f(u)+\sgn(\gamma)\psi, -|\gamma| p, -|\gamma| w)^\top.
\end{align*}
Then, we  directly compute 
\begin{align*}
   \nabla_{\mathbf{U}} E^{\alpha}(\mathbf{U})^\top\mathbf{Df(U)}&=(\sgn(\gamma) f(u),  \psi/\alpha, \beta w, \gamma p) \left(
\begin{array}{cccc}
	\sgn(\gamma) f'(u) & \sgn(\gamma)   & 0 & 0 \\
	\alpha f'(u) & \alpha \sgn(\gamma)  & 0 & 0 \\
	0 & 0 & 0 & -v_\beta^2 \\
	0 & 0 & -1 & 0 \\
\end{array}
\right)\\
&=\big(f'(u)(f(u) +\psi), f(u)+\sgn(\gamma) \psi, -|\gamma| p, -|\gamma| w\big)\\
&= \nabla_{\mathbf{U}} Q^{\alpha}(\mathbf{U})^\top.
\end{align*}
Hence, the pair $(E^{\alpha}, Q^{\alpha})$ satisfies the condition \eqref{entropy}. Due to the Assumption \ref{ass_fmonotone}, the entropy $E^{\alpha}$ is strictly convex, and the lemma is proven.
\end{proof}
In passing, we note that the entropy/entropy pair $(E^{\alpha}, Q^{\alpha})$  is asymptotically consistent with the pair $(E, Q)$ for \eqref{eq:dispersive} from \eqref{gradientflow} in the sense that  $(E^{\alpha}, Q^{\alpha})$ collapses to $(E, Q)$  when we let $\alpha \to \infty$.\\ 
One can readily check that the balance term $\mathbf b$  in \eqref{eq:homogeneous_system}
satisfies 
\[
\nabla_{\mathbf{U}} E^{\alpha}(\mathbf{U})\cdot \mathbf{b(U)}= 0
\]
for all $\mathbf{U}\in \mathbb{R}^4$. This property and the  compatibility condition \eqref{entropy} imply that any smooth solution $\mathbf{U}$ of \eqref{hyperbolic_system} is entropy conservative, i.e., 
\[{E^{\alpha}(\mathbf{U})}_t+{Q^{\alpha}(\mathbf{U})}_x  =   \nabla_{\mathbf{U}} E^{\alpha}(\mathbf{U})\cdot \mathbf{b(U)} = 0 \mbox{ in $I_T$}.
\]
In other words, the balance term does not dissipate the energy $E^{\alpha}$. 
For this reason, we cannot directly apply any standard results like \cite{yong2004entropy} to deduce global well-posedness of classical solutions (for at least appropriate initial data).  Therefore, we have to account for discontinuous solutions and introduce a 
weak solution concept for the hyperbolic system \eqref{hyperbolic_system}.
\begin{definition}\label{weak_soln}
 We call a function $\mathbf{U} \in L^\infty_{\mathrm{loc}}\big(I_T;\,\mathbb{R}^4\big)$ a weak solution of the system \eqref{hyperbolic_system} with initial data 
$\mathbf{U}_0 \in L^\infty(I;\mathbb{R}^4)$, if we have for each vector-valued function $  {\bm \varphi} \in C_0^{\lip} \big( I_T;\,\mathbb{R}^4\big)$ the identity 
\begin{equation*}
    \int_{0}^{T}\int_{I} 
    \Big( \mathbf{U}\cdot {\bm \varphi}_t
          + \mathbf{f}(\mathbf{U})\cdot {\bm \varphi}_x
          + \mathbf{b}(\mathbf{U})\cdot   {\bm \varphi}\Big)\,dx\,dt
    + \int_{I} \mathbf{U}_0\cdot {\bm \varphi}(\cdot, 0)\,dx = 0.
\end{equation*}
\end{definition}
In view of explicit entropy/entropy-flux pairs of \eqref{hyperbolic_system}, we define weak entropy solutions of \eqref{hyperbolic_system} as follows.
\begin{definition}\label{entropy_soln}
We call a function $\mathbf{U}\in {L}_{\mathrm{loc}}^{\infty}(I_T; ~\mathbb{R}^4 )$ a weak  entropy solution of the system \eqref{hyperbolic_system} associated with the entropy \eqref{convex-entropy} if it is a weak solution and if 
\begin{align}\hspace*{-0.25cm}
    \int_{0}^{T}\int_{I} \Big(E^{\alpha}(\mathbf{U})\varphi_t+Q^{\alpha}(\mathbf{U})\varphi_x+\nabla_{\mathbf{U}}E^{\alpha}(\mathbf{U})\cdot \mathbf{b(U)}\varphi\Big) \, dx\, dt+\int_{I} E^{\alpha}(\mathbf{U}_0)\varphi(\cdot, 0)\, dx\geq 0. 
\end{align}
holds for all non-negative functions $\varphi \in C^{\lip}_0( I_T)$.
\end{definition}
\subsubsection{Least action principle and energy conservation: Hamiltonian structure of the system \eqref{hyperbolic_system}}
In this section, we use the principle of classical mechanics to show that the entropy of the hyperbolic system \eqref{hyperbolic_system} is a Hamiltonian with respect to an augmented Lagrangian. In particular, the entropy \eqref{convex-entropy} can then be interpreted as the energy of the approximate system \eqref{hyperbolic_system}.\\ 
First, we recall that the dispersive equation \eqref{eq:dispersive} is derivable from a least action principle (see \S 11.7 in \cite{whitham2011linear}), as it can be considered as a generalized KdV equation. Indeed we consider the scalar potential $\eta= \eta(x,t)$ such that ${u}(x,t) = \eta_x(x,t)$ and define the Lagrangian $\mathcal{L}$ as
\begin{equation}
	\mathcal{L}[\eta_t,\eta_x,\eta_{xx}] := \int_I \prn{\sgn(\gamma)\left(\frac{1}{2}\eta_{x}\eta_{t} + F(\eta_x)\right) + \frac{|\gamma|}{2} \prn{\eta_{xx}}^2} dx, 
    \label{eq:org_Lagrangian}
\end{equation}
and then apply Hamilton's principle for $t_1,t_2\in [0,T]$ to the action 
\begin{equation}
	\mathcal{S}[\eta_t,\eta_x,\eta_{xx}] = \int_{t_0}^{t_1} \mathcal{L}[\eta_t,\eta_x,\eta_{xx}] \, dt. 
\end{equation}
This yields the Euler-Lagrange equation
\begin{equation}
	\prn{\pd{\mathcal{L}}{\eta_t}}_t + \prn{\pd{\mathcal{L}}{\eta_x}}_x - \prn{\pd{\mathcal{L}}{\eta_{xx}}}_{xx}= \sgn(\gamma)\prn{\eta_{xt} + \prn{f(\eta_x)}_x} - |\gamma| \eta_{xxx} = 0,
\end{equation}
which corresponds to \eqref{eq:dispersive} after substituting ${u}=\eta_x$.\\
We point out that the analogous least action principle can be obtained for the hyperbolic approximation in its wave propagation form \eqref{eq:waveform}. We prove this in the following lemma.
\begin{proposition}\label{hamiltonian}
The hyperbolic system \eqref{hyperbolic_system} can be formulated as a Hamiltonian system with an augmented Lagrangian of the form
\begin{equation}
	\mathcal{L}_\alpha[\eta_t,\eta_x,c,c_t,c_{x}] := \int_I \prn{\sgn(\gamma)\left(\frac{1}{2}\eta_x\eta_t + F(\eta_x)\right)+ \frac{|\gamma|c_{x}^2}{2}  + \frac{\alpha}{2}\prn{\eta_x - c}^2  - \frac{\beta}{2}c_t^2} dx. 
    \label{eq:lagrangian_aug}
\end{equation}	
Moreover, the Hamiltonian density with respect to the augmented Lagrangian \eqref{eq:lagrangian_aug} satisfies 
\begin{align}
    \mathcal{H}_{\alpha}[\eta_t,\eta_x,c,c_t,c_{x}]=-\int_{I}E^{\alpha}(\mathbf{U})\, dx,
\end{align}
where $E^{\alpha}(\mathbf{U})$ is the entropy of the hyperbolic system \eqref{hyperbolic_system} defined in \eqref{convex-entropy}.
\end{proposition}
\begin{proof}
Indeed, let us consider the so-called augmented Lagrangian
\begin{align*}
	\mathcal{L}_\alpha[\eta_t,\eta_x,c,c_t,c_{x}] = \int_I \prn{\sgn(\gamma)\left(\frac{1}{2}\eta_x\eta_t + F(\eta_x)\right)+ \frac{|\gamma|c_{x}^2}{2}  + \frac{\alpha}{2}\prn{\eta_x - c}^2  - \frac{\beta}{2}c_t^2} dx, 
\end{align*}	
which is built from the Lagrangian \eqref{eq:org_Lagrangian} by substituting $\eta_{xx}$ by $c_x$ and adding two relaxation terms proportional to $\alpha$ and $\beta$, respectively. Then, applying Hamilton's principle to the Lagrangian \eqref{eq:lagrangian_aug}, under the variations of $\eta$ and $c$ respectively, yields the two Euler-Lagrange equations
\begin{subequations}
\begin{align}
\prn{\pd{\mathcal{L_\alpha}}{\eta_t}}_t + \prn{\pd{\mathcal{L}_\alpha}{\eta_x}}_x &= \sgn(\gamma)\prn{\eta_{xt} + f(\eta_x)_x} + \alpha (\eta_x-c)_x = 0, 
\\ -\pd{\mathcal{L}_\alpha}{c}+\prn{\pd{\mathcal{L}_\alpha}{c_t}}_t + \prn{\pd{\mathcal{L}_\alpha}{c_x}}_x  &= -\beta c_{tt}  + |\gamma| c_{xx} + \alpha (\eta_x-c) = 0, 
\label{eq:EL_aug}
\end{align}
\end{subequations}
which again, after substituting $\eta_x =u$ reduce to the system \eqref{eq:waveform}.\\
Now the Hamiltonian density $\mathcal{H}_{\alpha}$ is defined by the partial Legendre transform of the Lagrangian density, and is given by
\begin{align*}
    \mathcal{H}_{\alpha}[\eta_t,\eta_x,c,c_t,c_{x}]=\pd{\mathcal{L_\alpha}}{\eta_t} \eta_t+\pd{\mathcal{L}_\alpha}{c_t} c_t-\mathcal{L_\alpha}[\eta_t,\eta_x,c,c_t,c_{x}].
\end{align*}
Using \eqref{eq:lagrangian_aug}, $\eta_x=u$ and in view of \eqref{convex-entropy}, we have 
\begin{align*}
    \mathcal{H}_{\alpha}[\eta_t,\eta_x,c,c_t,c_{x}]=-\int_{I}E^{\alpha}(\mathbf{U})\, dx.
\end{align*}
The lemma is then proven.
\end{proof}
Lemma \ref{hamiltonian} shows that the system \eqref{hyperbolic_system} also preserves the Hamiltonian and Lagrangian structure of the dispersive equation \eqref{eq:dispersive} as $\alpha\rightarrow \infty$, which is not possible in many relaxation systems. In particular, a Lagrangian interpretation is not available for the 
parabolic approximation from the work of Corli\&Rohde \cite{corli2014parabolic}. Therefore, we 
consider the hyperbolic approximation   \eqref{hyperbolic_system}, the more consistent choice if it is about the approximation of the purely dispersive equation 
\eqref{eq:dispersive}.

\subsection{Second-order approximations of the diffusive-dispersive equation \eqref{main}}\label{sec:main2}
We focus on the complete diffusive-dispersive equation, that is 
\eqref{main} with $\eps>0$ and $\gamma\in \mathbb{R}\setminus \{0\}$.
Thus, we  consider the  equation
\begin{equation}
	u_t + 
    {f(u)}_x = \eps u_{xx} + \gamma u_{xxx},~~\rm{in}~I_T,
	\label{eq:diffusivedispersive}
\end{equation}
subject to periodic boundary conditions and the 
initial condition
\begin{equation}
	u(\cdot,0) = u_0 \mbox{ in } I.
	\label{eq:diffusivedispersive_ini}
\end{equation}
The energy structure of the \eqref{eq:diffusivedispersive} is similar to the energy as defined in \eqref{energy}. Moreover, classical results for local and global well-posedness of classical solutions are available in $H^s$-spaces. In particular, we refer the interested reader to the work of Gallego\&Pazato \cite{Gallegogkdv} for a systematic treatment of generalized KdV-Burgers equations.\\
To approximate the diffusive-dispersive equation \eqref{eq:diffusivedispersive}, we include a viscoelastic damping operator \cite{ikehata2017critical} instead of the pure wave operator in \eqref{eq:waveform}. Precisely, we propose the following approximate system  in $I_T $ for the diffusive-dispersive equation \eqref{eq:diffusivedispersive}: 
\begin{equation}
   \begin{aligned}
u_t + f(u)_x + \sgn(\gamma)\alpha \prn{u - c}_x&=\varepsilon u_{xx}, \\[1.1ex]
    \beta (c_{tt}-\varepsilon c_{xxt})-|\gamma| c_{xx} &= \alpha \prn{u - c}
\end{aligned}
 \label{eq:damped_wave_system}
\end{equation}
With the change of variables 
\[
 w=c_t-\varepsilon c_{xx}, \quad p=c_x, \quad \psi=\alpha (u-c),
\]
one can easily convert the system \eqref{eq:damped_wave_system} into the  mixed- but lower-order  system
\begin{equation}
\begin{aligned}
	u_t + \prn{f(u) + \sgn(\gamma)\psi}_x &=\varepsilon u_{xx}, \\
	\dfrac{1}{\alpha}\psi_t+\prn{f(u)+\sgn(\gamma)  \psi}_x &= - w + \dfrac{\varepsilon}{\alpha} \psi_{xx},\\
		\beta w_{t}-|\gamma| p_{x} &=\psi , \\
	p_t - w_x&=\varepsilon p_{xx}
\end{aligned} \label{hyperbolic_parabolic_system}  \qquad \quad \mbox{in } I_T.
\end{equation}
The system \eqref{hyperbolic_parabolic_system} is complemented with 
the prepared initial data
\begin{equation}
u\rel(\cdot,0)= u_0,   \,  \psi\rel(\cdot,0)= -|\gamma| u_{0,xx},\,
w\rel(\cdot,0) =  - {f(u_0)}_x - |\gamma| u_{0,xxx}, \, 
p\rel(\cdot,0) =   u_{0,x} \mbox{ in }I.
\label{hyperbolic_parabolic_system_ini}
\end{equation}
In what follows, we prove that the entropy/entropy-flux pairs for the first-order system \eqref{hyperbolic_system} are compatible with the system \eqref{hyperbolic_parabolic_system} as well. 

\subsection{Qualitative properties of the second-order approximation}

\subsubsection{Hyperbolic-parabolic structure,  energy dissipation, and local well-posedness}
Let us again omit  in this section the superscript $(.)\rel$ on the  unknowns  and  write ${\mathbf U} = (u, \psi, w, p)^T$ for  $ {\mathbf U}\rel = (u\rel, \psi\rel, w\rel, p\rel)^T\in  \setR^4$.\\ 
With the notations from   \eqref{eq:homogeneous_system}  and the 
positive semi-definite matrix 
\begin{equation*}
{\mathbf D} = 
\begin{pmatrix}
\varepsilon &0 &0 & 0 \\
0 &\varepsilon &0 & 0 \\
0 &0 &0 & 0 \\
0 &0 &0 & \varepsilon 
\end{pmatrix}
\end{equation*}
we write the  system \eqref{hyperbolic_parabolic_system} for some smooth solution $\mathbf U$ in the form 
\begin{equation}
\mathbf{U}_t + {{\mathbf f}(\mathbf{U})}_x = \mathbf{S}\mathbf{U}  +    {\mathbf D} \mathbf{U}_{xx}\Longleftrightarrow
	\mathbf{U}_t + \mathbf{Df}(\mathbf{U})\mathbf{U}_x = 
    \mathbf{S}\mathbf{U} + {\mathbf D} \mathbf{U}_{xx}.
    \label{eq:dissipative_system}
\end{equation} 
From \eqref{eq:dissipative_system} we see that \eqref{hyperbolic_parabolic_system} is a (degenerate) hyperbolic-parabolic system. In the next step we show that for Assumption \ref{ass_fmonotone} imposed, smooth solutions of \eqref{hyperbolic_parabolic_system} dissipate the entropy $E^{\alpha}$ from \eqref{convex-entropy}. 
Recall that smooth solutions of the hyperbolic system \eqref{hyperbolic_system} conserve 
$E^{\alpha}$.

\begin{Lemma}[Entropy dissipation for \eqref{hyperbolic_parabolic_system}]\label{zero_order_estimates}
Let Assumption \ref{ass_fmonotone} hold and let 
$\mathbf U$ be a smooth solution of the hyperbolic-parabolic system \eqref{hyperbolic_parabolic_system} in $I_T$.\\
Then we have for   $E^{\alpha}$ from \eqref{entropy} in $(0,T)$ the  estimate
\begin{equation}\label{parabolic_estimate}
\begin{array}{rcl}
\dfrac{d}{dt} \int_I E^{\alpha}({\mathbf U)}\, dx
&=&\dfrac{d}{dt} \int_{I}\left( \sgn(\gamma)F(u)+\dfrac{|\gamma| p^2}{2}+\dfrac{\psi^2}{2\alpha}+\dfrac{\beta w^2}{2}\right)\, dx\\[1.9ex]
&=& -\varepsilon \displaystyle\int_{I}\left(\sgn(\gamma)f'(u) u_x^2+|\gamma| p_x^2+\dfrac{\psi_x^2}{\alpha}\right) \,dx\\[1.2ex]
&\le & 0.
\end{array}
\end{equation}
\end{Lemma}
\begin{proof}
We directly take the time derivative of $E^{\alpha}({\mathbf U})$ and compute
\begin{align*}
    \dfrac{d}{dt}E^{\alpha}({\mathbf U})= \sgn(\gamma)f(u) u_t+\beta w w_t+\dfrac{\psi}{\alpha} \psi_t +|\gamma| p p_t.
\end{align*}
Then, in view of \eqref{hyperbolic_parabolic_system}, we have
\begin{align*}
    \dfrac{d}{dt}E^{\alpha}({\mathbf U})=-\dfrac{d}{d x}\left(\dfrac{f(u)^2}{2}+f(u) \psi+\sgn(\gamma)\dfrac{\psi^2}{2}- |\gamma| pw\right)+\varepsilon\left(\sgn(\gamma)f(u) u_{xx}+|\gamma| pp_{xx}+\dfrac{\psi}{\alpha} \psi_{xx}\right).
\end{align*}
Integrating both sides with respect to $x$ and using the periodic boundary conditions, we get the statement of the lemma. 
\end{proof}
In view of Assumption \ref{ass_fmonotone}, the entropy in \eqref{convex-entropy} is strictly convex and therefore the local well-posedness of the hyperbolic-parabolic system \eqref{hyperbolic_parabolic_system} for the class of flux functions satisfying Assumption \ref{ass_fmonotone} can be obtained using the classical theory of symmetrizable hyperbolic-parabolic systems due to Kato \cite{MR390516} or Kawashima \cite{kawashima1984systems}. To be precise, we have the following.
\begin{lemma}[Local well-posedness of the initial value problem for system \eqref{hyperbolic_parabolic_system}]\label{local_existence}
    Let  the function $\mathbf{U}_0=(u_0, \psi_0, w_0, p_0)^\top$ be given such that 
\[
\mathbf{U}_0 \in ({H}^{s}(I))^4, \quad s > 3/2,
\]
holds. Suppose that the flux $f$ satisfies  Assumption \ref{ass_fmonotone}.\\
Then, there exists  for all $\alpha,\beta>0$ and $\gamma\in \mathbb{R}\setminus \{0\}$ a finite time \(T^* \in (0, \infty)\) depending on ${\lVert\mathbf{U}_0 \rVert}_{({H}^{s}(I))^4}$, such that the initial value  problem for \eqref{hyperbolic_parabolic_system} with $\mathbf{U}(\cdot,0)= \mathbf{U}_0$ 
admits a unique classical solution $\mathbf U$, which satisfies
\[
\mathbf{U}(\cdot, t) \in C\!\left([0, T^*]; ({H}^{s}(I))^4\right) 
\cap C^1\!\left([0, T^*]; (H^{s-1}(I))^4\right), \quad s>3/2.
\]
\end{lemma}
\begin{remark}\label{remark_non_convex}
\begin{enumerate}
    \item Note that the entropy dissipation for our system \eqref{hyperbolic_parabolic_system} in Lemma \ref{zero_order_estimates} coincides with the energy dissipation for the elliptic approximation \eqref{eq:scrn_poiss} (see Lemma 3.3 in \cite{corli2012singular}). In particular, this shows that the proposed system \eqref{hyperbolic_parabolic_system} does not introduce any additional diffusion into the dynamics. In contrast, the parabolic approximation in \cite{corli2014parabolic} exhibits stronger diffusive effects than the proposed system \eqref{hyperbolic_parabolic_system} (see Lemma 3.5 in \cite{corli2014parabolic}).
\item The condition on the convexity of $\sgn(\gamma)F(u)$ for the diffusive-dispersive case can be further relaxed. Assume that $\sgn(\gamma)f'(u)\geq -\bar{f}$ for some $\bar{f}>0$. Then one can choose a number $\kappa$ such that $\kappa+\bar{f}>\delta>0$. In that case, we can have a similar energy inequality with a shift in $\sgn(\gamma)F(u)$ given by $\sgn(\gamma)F(u)+\dfrac{\kappa u^2}{2}$, which still makes the energy convex. However, this can only work for the diffusive-dispersive equation, as it gives enough dissipation to absorb the residual terms for a fixed $\varepsilon>0$.   
\item Note that the system \eqref{hyperbolic_system} preserves the energy of the limit equation \eqref{eq:diffusivedispersive} regardless of the choice of flux functions. In this work, we restrict ourselves to the class of flux functions satisfying Assumption~\ref{ass_fmonotone}, as the convergence analysis in Section~\ref{sec: convergence} relies on the uniform convexity of the entropy $E_\alpha$. Nevertheless, the numerical examples in Section~\ref{sec: Numerics} indicate that the approximate system \eqref{hyperbolic_system} provides accurate approximations even for equations, where the flux function does not satisfy Assumption~\ref{ass_fmonotone}.
\end{enumerate}
\end{remark}
\subsubsection{Global well-posedness of classical solutions of \eqref{hyperbolic_parabolic_system} for prepared initial data}
In this section, we prove that the local solutions developed in Lemma \ref{local_existence} can be extended to $(0, \infty)$ as long as the initial data is in a sufficiently small neighborhood of a constant equilibrium point $\mathbf{U}^*=(u^*, 0, 0, 0)^\top$. In particular, we prove that the system \eqref{hyperbolic_parabolic_system} belongs to the class of partially diffusive systems satisfying the Shizuta-Kawashima (SK) condition \cite{shizuta1985systems, kawashima1984systems}. Precisely, we prove the following.
\begin{theorem}\label{global_wellposedness}
  Let $\alpha,\beta>0$ and $\gamma\in \mathbb{R}\setminus \{0\}$ be given. Suppose that the flux $f$ satisfies  Assumption \ref{ass_fmonotone}. Then there exists a $\Lambda>0$ such that for all $\mathbf{U}_0=(u_0, \psi_0, w_0, p_0)^\top\in ({H}^{s}(I))^4$, $s \geq 2, ~s\in \mathbb{N}$, with
  \[
 \left(
 {\lVert  u_0-u^*\rVert}_{{H}^{s}(I)}^2+\dfrac{1}{\alpha}{\lVert  \psi_0\rVert}_{{H}^{s}(I)}^2+\beta{\lVert w_0\rVert}_{{H}^{s}(I)}^2+|\gamma|{\lVert p_0\rVert}_{{H}^{s}(I)}^2
 \right)<\Lambda.
  \]
the initial value  problem for \eqref{hyperbolic_parabolic_system} has a unique classical solution $\mathbf{U}$ with $\mathbf{U}(\cdot,0)= \mathbf{U}_0$. It satisfies 
  \begin{align}
     \mathbf{U}-\mathbf{U}^*\in C([0, \infty); ({H}^{s}(I))^4)\cap  C^1([0, \infty); ({H}^{s-1}(I))^4).
  \end{align}
Furthermore, we have for a positive constant $C=C(\Lambda)$ the 
estimate
\begin{equation}\label{energy-estimates}
\begin{aligned}
\sup_{0\leq t<\infty}&\left({\lVert u-u^*\rVert}_{{H}^{s}(I)}^2+\dfrac{1}{\alpha}{\lVert \psi\rVert}_{{H}^{s}(I)}^2+\beta {\lVert w\rVert}_{{H}^{s}(I)}^2+|\gamma|{\lVert p\rVert}_{{H}^{s}(I)}^2\right)\\
&+ \varepsilon \int_{0}^{t}\left({\lVert  u-u^*\rVert}_{{H}^{s+1}(I)}+|\gamma|{\lVert  p\rVert}_{{H}^{s+1}(I)}+\dfrac{1}{\alpha}{\left\lVert \psi\right\rVert}_{{H}^{s+1}(I)}\right)\, d \tau\\
&\leq C(\Lambda)\bigg({\lVert u_0-u^*\rVert}_{{H}^{s}(I)}^2+\dfrac{1}{\alpha}{\lVert \psi_0\rVert}_{{H}^{s}(I)}^2+\beta {\lVert w_0\rVert}_{{H}^{s}(I)}^2+|\gamma|{\lVert p_0\rVert}_{{H}^{s}(I)}^2\bigg).
\end{aligned}
\end{equation}
\end{theorem}
\begin{proof}
Since the system \eqref{hyperbolic_parabolic_system} is endowed with an entropy/entropy-flux pair, it can be converted into a symmetric hyperbolic-parabolic form by using the Hessian of the entropy $E^{\alpha}(\mathbf{U})$ as its symmetrizer. By multiplying the equation \eqref{eq:dissipative_system} with $\nabla^2_{\mathbf{U}} E^{\alpha} (\mathbf{U})$, we obtain the symmetric form of \eqref{eq:dissipative_system} as
\begin{align}\label{symmetric_form} \mathbf{S}_0(\mathbf{U})\mathbf{U}_t+\mathbf{S}_1(\mathbf{U})\mathbf{U}_x-\mathbf{S_2(U)}\mathbf{U}_{xx}=\mathbf{S_3(U)\mathbf{U}},
 \end{align}
 where
\begin{align*}
    \mathbf{S_0(U)}&=\nabla^2_{\mathbf{U}} E^{\alpha} (\mathbf{U})=\begin{pmatrix}
\sgn(\gamma)f'(u) &0 &0 & 0 \\
0 & 1/\alpha &0 & 0 \\
0 &0 &\beta & 0 \\
0 &0 &0 & |\gamma| 
\end{pmatrix},\\
\mathbf{S_1(U)}&=\nabla^2_{\mathbf{U}} E^{\alpha}(\mathbf{U})\mathbf{Df(U)}=\begin{pmatrix}
(f'(u))^2 & f'(u) &0 & 0 \\
f'(u) & \sgn(\gamma) &0 & 0 \\
0 &0 &0 & -|\gamma| \\
0 &0 &-|\gamma| &0 
\end{pmatrix},\\
\mathbf{S_2(U)}&=\nabla^2_{\mathbf{U}} E^{\alpha} (\mathbf{U})\mathbf{D}=\begin{pmatrix}
\varepsilon \sgn(\gamma) f'(u) &0 &0 & 0 \\
0 & \varepsilon/\alpha &0 & 0 \\
0 &0 &0 & 0 \\
0 &0 &0 & |\gamma| \varepsilon
\end{pmatrix}, \\
\mathbf{S_3(U)}&=\nabla^2_{\mathbf{U}} E^{\alpha} (\mathbf{U})\mathbf{S}=
\begin{pmatrix}
0 &0 &0 & 0 \\
0 & 0 &-1 & 0 \\
0 &1 &0 & 0 \\
0 &0 &0 & 0
\end{pmatrix}.
\end{align*}
Clearly, the system \eqref{symmetric_form} belongs to the general class of symmetric partially diffusive hyperbolic-parabolic systems with a skew-symmetric source matrix. \\
We first verify that the flux and the diffusion matrix of the system \eqref{symmetric_form} satisfy the Shizuta-Kawashima (SK) conditions \cite{shizuta1985systems, kawashima1984systems}. The kernel of the diffusion matrix $\mathbf{S_2(U)}$ is given by
\[
\ker \mathbf{S_2(U)}
= \big\{ (0,0,w,0)^{\top} : w \in \mathbb{R} \big\}
= \operatorname{span}\big\{ (0,0,1,0)^{\top} \big\}.
\]
A straightforward computation shows that no eigenvectors of $\mathbf{S_1(U)}$ lies in $\ker \mathbf{S_2(U)}$ and thus we have for all eigenvalues $\lambda$ of  $\mathbf{S_1(U)}$ 
\[
\ker \mathbf{S_2(U)}\cap \ker (\mathbf{S_1(U)}-\lambda \mathbf{I})=\{0\}.
\]
This shows that the system \eqref{hyperbolic_parabolic_system} or equivalently \eqref{eq:dissipative_system} belongs to the partially diffusive systems satisfying the SK coupling conditions. Moreover, the source term does not contribute to the higher-order energy dissipation (or increment). To be precise, it is easy to see that the source term vanishes after multiplying by energy multipliers for $H^s$-estimates ($s=0, 1, 2, \ldots$), or in other words
\[
\partial_x^s (\mathbf{U}-\mathbf{U}^*)^\top \mathbf{S_3(U)}\partial_x^s (\mathbf{U}-\mathbf{U}^*)=0, \quad s=0, 1, 2, \ldots.
\]
This, in particular, allows us to follow the classical theory for the global well-posedness of smooth solutions of partially diffusive systems (see e.g. Theorem 4.3 in \cite{kawashima1984systems}). Therefore, the results of the theorem are a direct consequence of the classical theory due to Kawashima \cite{kawashima1984systems}.
\end{proof}
\begin{remark}
 There are recent results on the local and global well-posedness for partially diffusive systems where the initial data need less regularity and, e.g., belong to some Besov space instead. We refer the interested reader to the recent work of Adogbo\&Danchin \cite{Adogbo_JDE} for more details. However, for our analysis, we work with $H^s$-spaces. 
\end{remark}

\section{Hyperbolic approximations and asymptotics}\label{sec: convergence}
In this section, we first apply a Chapman–Enskog expansion on the variables of the system \eqref{eq:waveform} and \eqref{eq:damped_wave_system} to provide a formal justification of the consistency of the proposed system(s) as $\alpha \rightarrow \infty$. The Chapman–Enskog expansion provides an asymptotic expansion for the solutions of \eqref{eq:waveform} and \eqref{eq:damped_wave_system} and identifies \eqref{eq:dispersive} and \eqref{eq:diffusivedispersive} as the limit equation(s), respectively. However, this argument is not rigorous at the level of entropy solutions. To obtain a rigorous convergence result, we complement the formal expansion with a relative entropy framework to the systems \eqref{hyperbolic_system} and \eqref{hyperbolic_parabolic_system}. Precisely, we prove the convergence of the solutions of approximate systems to the classical solutions of \eqref{eq:diffusivedispersive} for $\varepsilon\geq 0$ and the flux functions satisfying Assumption \ref{ass_fmonotone}.\medskip\\
Throughout this section, we denote the solutions of the approximate systems \eqref{hyperbolic_system} and \eqref{hyperbolic_parabolic_system} as $\mathbf{U}\rel=(u\rel, \psi\rel, w\rel, p\rel)^\top$ and the solution of the diffusive-dispersive equation \eqref{eq:diffusivedispersive} by $\mathbf{{U}}=({u}, {\psi}, {w}, {p})^\top=({u}, -|\gamma| {u}_{xx}, {u}_t, {u}_x)^\top$ for $\varepsilon\geq 0$. Based on the proposed systems \eqref{hyperbolic_system} and \eqref{hyperbolic_parabolic_system}, we divide our convergence analysis into two different sections.  
\subsection{Dispersive asymptotics  for the first-order hyperbolic system \eqref{hyperbolic_system}}
We  discuss  the asymptotics for $\alpha \to \infty$ of the sequence ${\{ {\mathbf U}^\alpha\}}_{\alpha >0}$ solving  \eqref{hyperbolic_system} for prepared initial data \eqref{hyperbolic_parabolic_system_ini} and periodic boundary conditions. We will show in Theorem \ref{relative_entropy_theorem_1}
that the sequence converges to a solution of the dispersive equation 
\eqref{eq:dispersive} subject to the initial datum from \eqref{eq:dispersive_ini}. \\
To motivate a scaling of $\beta$ in terms of the parameter $\alpha$ we 
consider Chapman-Enskog expansions. Precisely we choose $\beta = \alpha^{-1}$ and assume that the solutions of the system \eqref{eq:waveform} admit a smooth expansion in the parameter $\alpha$ such that 
\[
\begin{array}{rcl}
u^\alpha(x,t) &= &u^0(x,t) + \alpha^{-1} u^1(x,t) + \alpha^{-2} u^2(x,t) + \mathcal{O}(\alpha^{-3}), \\
c^\alpha(x,t) &=& c^0(x,t) + \alpha^{-1} c^1(x,t) + \alpha^{-2} c^2(x,t) + \mathcal{O}(\alpha^{-3}).
\end{array}
\]
Inserting these expansions in $\eqref{eq:waveform}_2$ allows to write
\begin{equation}
 \alpha \prn{u^0 - c^0} + \prn{u^1 - c^1 + |\gamma| c^0_{xx}} + \alpha^{-1} (u^2-c^2 -c^0_{tt}+ |\gamma| c^1_{xx})  =  \mathcal{O}(\alpha^{-2}).
\label{eq:series_1}
\end{equation}
Then, successively equating the series coefficients to zero in \eqref{eq:series_1} leads to
\begin{equation}
u^0 - c^0 =0, \quad u^1 - c^1 = -|\gamma| c^0_{xx}, \quad  u^2 - c^2 = c^0_{tt}-|\gamma| c^1_{xx},
\end{equation}
and therefore
\begin{align*}
\alpha\prn{u^\alpha-c^\alpha} &= -|\gamma| c^0_{xx} +\alpha^{-1}(c^0_{tt}-|\gamma| c^1_{xx}) + \mathcal{O}(\alpha^{-2}) \\
    &= -|\gamma| u^0_{xx} + \alpha^{-1} (u^0_{tt}-|\gamma|\prn{u^1 + |\gamma| u^0_{xx}}_{xx}) + \mathcal{O}(\alpha^{-2}) \\ 
    &= -|\gamma| \prn{u^0 + \alpha^{-1} u^1}_{xx} + \alpha^{-1}(u^0_{tt} -\gamma^2 u^0_{xxxx} + \mathcal{O}(\alpha^{-2}) \\ 
& = -|\gamma| u^\alpha_{xx} +\alpha^{-1} (u\rel_{tt} -\gamma^2 {u}^\alpha_{xxxx}) + \mathcal{O}(\alpha^{-2}).
\end{align*}
Using the last relation in  $\eqref{eq:waveform}_1$ finally yields
\begin{equation}
     u\rel_t + f(u\rel)_x = \sgn(\gamma)|\gamma| u\rel_{xxx}  + \alpha^{-1} (\gamma^2u\rel_{xxxxx}-u\rel_{ttx} ) + \mathcal{O}(\alpha^{-2}).
\end{equation}
Thus, given the suggested scaling $\beta = \alpha^{-1}$, the leading order error term with respect to the dispersive equation \eqref{eq:dispersive} is first-order in $\alpha^{-1}$.
\bigskip\\
In what follows, we prove that  weak entropy solutions of \eqref{hyperbolic_system}-\eqref{IC_BC_Hyp} converge to the solutions of the dispersive equation \eqref{eq:dispersive} when the 
parameter  $\alpha\rightarrow \infty$ as our first main result.
\begin{theorem}[Main result I]\label{relative_entropy_theorem_1}
 Let   ${u}_0\in H^6(I)$, and suppose that Assumption \ref{ass_fmonotone} holds. We assume  that 
 ${u}\in {L}^{2}(0, T;H^3(I))$ with ${u}_{tt}\in {L}^{2}(I_T)$ and ${\lVert{u}\rVert}_{L^\infty(0, T;H^4(I))}<\infty$
 is a classical solution of \eqref{eq:dispersive}, \eqref{eq:dispersive_ini}. Consider for $\beta= \alpha^{-1}$, a 
 sequence   ${\{\mathbf{U}^\alpha=(u^\alpha, \psi^\alpha, w^\alpha, p^\alpha)^\top\}}_{\alpha>0}$  
of  weak entropy solutions 
 of \eqref{hyperbolic_system}, \eqref{IC_BC_Hyp} in the sense of Definition \ref{entropy_soln}.\\
   Then there is a  constant $C= C(L,\delta,  u)\ge 0$ that is independent of $\alpha$ such that
    \begin{align}\label{convergence_rates}
        {\lVert u\rel-{u}\rVert}_{L^{\infty}\big(0, T; {L}^{2}(I)\big)}+ {\lVert p\rel-{u}_x \rVert}_{L^{\infty}\big(0, T; {L}^{2}(I)\big)}\le C\sqrt{\dfrac{1}{\alpha}}
    \end{align}
    holds.
\end{theorem}
\begin{proof}  
With $\psi =-|\gamma| {u}_{xx}$, $ w={u}_t $ and $ p = {u}_x$, we rewrite the dispersive limit equation \eqref{eq:dispersive} in terms of the approximate system \eqref{hyperbolic_system}, which leads to the system
\begin{equation}\label{reformulated_dispersion}
\begin{aligned}
{u}_t+f({u})_x+\sgn(\gamma){\psi}_x&=0,\\
    \dfrac{1}{\alpha}{\psi}_t+f({u})_x+\sgn(\gamma){\psi}_x&=-{w}-\dfrac{1}{\alpha} |\gamma| {u}_{xxt},\\
    \beta {w}_t-|\gamma| {p}_x&={\psi}+\beta {u}_{tt},\\
    {p}_t-{w}_x&=0.
\end{aligned}
\end{equation}
Define $ {\mathbf U} = ( u, \psi, w, p)^\top$ and fix 
the  initial datum $ {\mathbf U}_0 = ( u_0, -|\gamma|  u_{0, xx}, -f(u_0)_x-|\gamma| u_{0, xxx}, {u}_{0, x})^\top$. 
Since ${u}$ is a classical solution of the initial value problem for \eqref{eq:dispersive} with $  u (\cdot,0)= u_0$, the function
 $  {\mathbf U}$ 
is a classical solution of \eqref{reformulated_dispersion} with
${\mathbf U} (\cdot,0) =   {\mathbf U}_0$.\\
It is then straightforward to check 
that ${\mathbf U} $  satisfies for  
the entropy/entropy-flux  pair $(E^{\alpha}, Q^{\alpha})$
from Lemma \ref{lemma_entropy} 
 the integral identity 
\begin{align}\label{entrop_identity_dispersion} \hspace*{-0.3cm}
    \int_{0}^{T}\int_{I} \left(E^{\alpha}(\mathbf{{U}})\varphi_t+Q^{\alpha}(\mathbf{{U}})\varphi_x+\nabla_{\mathbf{{U}}}E^{\alpha}(\mathbf{{U}})\cdot \mathbf{\bar{b}({U})\varphi}\right) \, dx\, dt+\int_{I} E^{\alpha}(\mathbf{{U}}_0)\varphi(\cdot,0)\, dx= 0
\end{align}
 for all $\varphi \in C_{0}^{\lip}( I \times [0,T))$, $\varphi \ge 0$.
In \eqref{entrop_identity_dispersion} we used  
$\mathbf{\bar{b}({{\mathbf {{U}}}})}=\big(0, -\alpha{w}- |\gamma| {u}_{xxt}, {\beta}^{-1}{\psi}+{u}_{tt}, 0\big)^\top$.\\
Furthermore, since $\mathbf{U}^{\alpha}$ is a weak entropy solution of  \eqref{hyperbolic_system}, \eqref{IC_BC_Hyp}, we have
\begin{equation}\label{rel_entropy_identity_hyp_system}
\int_{0}^{T}\int_{I} \big(E^{\alpha}(\mathbf{U}\rel)\varphi_t+Q^{\alpha}(\mathbf{U}\rel)\varphi_x+\nabla_{\mathbf{U}}E^{\alpha}(\mathbf{U}\rel)\cdot \mathbf{b(U\rel)}\varphi\big) \, dx\, dt
 +\int_{I} E^{\alpha}(\mathbf{U}\rel_0)\varphi(\cdot,0)\, dx\geq 0. 
\end{equation}
Now, we recall that the  entropy $E^{\alpha}$ from \eqref{convex-entropy} is almost quadratic and define the associated  relative entropy $\mathcal{E}^{\alpha}$ for the approximate solution ${\mathbf U}\rel$   and the function $\mathbf{{U}}$
 according to 
\begin{equation}\label{rel_entropy_hyp_system} 
\begin{aligned}
\hspace*{-0.3cm}
    \mathcal{E}^{\alpha}(\mathbf{U}\rel| \mathbf{ U})&=E^{\alpha}(\mathbf{U}\rel)-E^{\alpha}(\mathbf{{U}})-\nabla_{\mathbf{U}}E^{\alpha}(\mathbf{{U}})(\mathbf{U}\rel-\mathbf{{U}})\\
    &=\sgn(\gamma)\big(F(u\rel)-F({u})-f({u})(u\rel-{u})\big)+\dfrac{|\gamma|}{2}(p\rel-{p})^2+\dfrac{ \beta}{2}  (w\rel-{w})^2+\dfrac{1}{2\alpha}(\psi\rel-{\psi})^2.
    \end{aligned}
\end{equation}
Recall that  $\sgn(\gamma)F$ is a uniformly convex function by Assumption \ref{ass_fmonotone}. Thus, there exists an  $\alpha$-independent  constant  $C_1>0$ such that
\begin{align}\label{relative_entropy_relation_main}
    C_1\left((u^\alpha-{u})^2+\dfrac{1}{\alpha}(\psi^\alpha-{\psi})^2+\beta(w^\alpha-{w})^2+(p^\alpha-{p})^2\right)\leq \mathcal{E}^{\alpha}(\mathbf{U}\rel|\mathbf{{U}}).
\end{align}
Now, we subtract \eqref{entrop_identity_dispersion} from \eqref{rel_entropy_identity_hyp_system} to obtain 
\begin{equation}\label{difference_idenity}
\begin{aligned}
    &\hspace*{-2cm}\int_{0}^{T}\int_{I} \big((E^{\alpha}(\mathbf{U}\rel)-E^{\alpha}(\mathbf{{U}}))\varphi_t+(Q^{\alpha}(\mathbf{U}\rel)-Q^{\alpha}(\mathbf{{U}}))\varphi_x\big) \, dx\, dt\\
+&\int_{0}^{T}\int_{I}\left(\nabla_{\mathbf{U}}E^{\alpha}(\mathbf{U}\rel)\cdot \mathbf{b(U\rel)}-\nabla_{\mathbf{{U}}}E^{\alpha}(\mathbf{{U}})\cdot \mathbf{{b}({U})}\right)\varphi \, dx\, dt\\
+&\int_{I} \left(E^{\alpha}(\mathbf{U}_0)-E^{\alpha}(\mathbf{{U}}_0)\right)\varphi(x, 0)\, dx\geq 0. 
\end{aligned}
\end{equation}
Moreover, since $\mathbf{{U}}$ is a classical solution of \eqref{reformulated_dispersion}, it must also be a solution of the system \eqref{reformulated_dispersion} in the distributional sense. If  we  choose  for  $\varphi \in C_0^{\lip}(I \times [0,T)  )$ the test function ${\bm \varphi}=\varphi  \nabla_{\mathbf{U}} E^{\alpha}(\mathbf{{U}})\in  C_0^{\lip}( I \times [0,T); \mathbb{R}^4)$ like in Definition \ref{weak_soln}, we obtain
 from the fact that $\mathbf{U}\rel$ is a weak solution, the identity
\begin{equation}\label{weak_relative}
\begin{aligned}
    &\hspace*{-0.5cm}\int_{0}^{T}\int_{I} (\mathbf{U}\rel-\mathbf{{U}})\cdot (\varphi \nabla_{\mathbf{U}} E^{\alpha}(\mathbf{{U}})_t+(\mathbf{f(U\rel)}-\mathbf{f({U})})\cdot(\varphi\nabla_{\mathbf{U}} E^{\alpha}(\mathbf{{U}}))_x\, dx\, dt\\
    +&\int_{0}^{T}\int_{I} (\mathbf{b(U\rel)}-\mathbf{{b}({U})})\cdot (\varphi \nabla_{\mathbf{U}} E^{\alpha}(\mathbf{{U}})) \, dx\, dt +\int_{I} (\mathbf{U}\rel_0-\mathbf{{U}}_0)\cdot (\varphi\nabla_{\mathbf{U}} E^{\alpha}(\mathbf{{U}}))( x, 0)\, dx= 0. 
\end{aligned}
\end{equation}
From the relations 
\[
\left(\nabla_{\mathbf{{U}}} E^{\alpha}(\mathbf{{U}})\right)_t=\mathbf{{U}}_t^\top \nabla^2_{\mathbf{{U}}} E^{\alpha}(\mathbf{{U}}) ,\quad \text{and}~ \left(\nabla_{\mathbf{{U}}} E^{\alpha}(\mathbf{{U}})\right)_x= \mathbf{{U}}_x^\top\nabla^2_{\mathbf{{U}}} E^{\alpha}(\mathbf{{U}}) ,
\]
we get after  subtracting \eqref{weak_relative} from \eqref{difference_idenity}  the identity 
\begin{equation}\label{difference_idenity_weak_form}
\begin{aligned}
    &\hspace*{-1cm}\mathcal{I}=\int_{0}^{T}\int_{I} \left(\mathcal{E}^{\alpha}(\mathbf{U}\rel|\mathbf{{U}}))\varphi_t+(\mathcal{Q}^{\alpha}(\mathbf{U}\rel|\mathbf{{U}}))\varphi_x\right) \, dx\, dt\\
+&\underbrace{\int_{0}^{T}\int_{I} \left(\mathbf{U}\rel-\mathbf{{U}})^\top \mathbf{{U}}_t^\top\nabla^2_{\mathbf{{U}}} E^{\alpha}(\mathbf{U}) +(\mathbf{f(U\rel)}-\mathbf{f({U}}))^\top\mathbf{{U}}_x^\top\nabla^2_\mathbf{U} E^{\alpha}(\mathbf{{U}}) )\right)\varphi\, dx\, dt}_{=:\mathcal{I}_1}\\
+&\underbrace{\int_{0}^{T}\int_{I}\left(\nabla_{\mathbf{U}}E^{\alpha}(\mathbf{U}\rel)-\nabla_{\mathbf{U}}E^{\alpha}(\mathbf{{U}})\right)\cdot \mathbf{b(U\rel)} \varphi\, dx\, dt}_{=:\mathcal{I}_2}+\int_{I} \mathcal{E}^{\alpha}(\mathbf{U}_0|\mathbf{{U}}_0)\varphi(x,0)\, dx\geq 0. 
\end{aligned}
\end{equation}
Here we used the relative entropy flux 
\[
\mathcal{Q}^{\alpha}(\mathbf{U}\rel| {\mathbf U})= Q^{\alpha}(\mathbf{U}^\alpha)-Q^{\alpha}(\mathbf{{U}})-\nabla_{\mathbf{U}}E^{\alpha}(\mathbf{{U}})\cdot (\mathbf{f}(\mathbf{U}^\alpha)-\mathbf{f({U})}), \]
which satisfies for an $\alpha$-independent constant $C_2\ge0$ the inequality
$|\mathcal{Q}^{\alpha}(\mathbf{U}\rel| {\mathbf U})|\leq C_2 |\mathcal{E}^{\alpha}(\mathbf{U}\rel| {\mathbf U})|.
$\\
Now we focus on the integrals $\mathcal{I}_1$ and $\mathcal{I}_2$ in \eqref{difference_idenity_weak_form}. First, consider the integral $\mathcal{I}_1$, which, in view of \eqref{reformulated_dispersion} writes as
\begin{equation}\label{I_1}
\begin{aligned}
\mathcal{I}_1&=\int_{0}^{T}\int_{I} \left(\left(\mathbf{{U}}_t\cdot(\mathbf{U}\rel-\mathbf{{U}}) +\mathbf{{U}}_x\cdot(\mathbf{f(U\rel)}-\mathbf{f({U}})) \right)^\top\nabla^2_{\mathbf{U}} E^{\alpha}(\mathbf{{U}})\right)\varphi\, dx\, dt\\
&=\int_{0}^{T}\int_{I} \left(\left(\mathbf{{U}}_x\cdot \left(\mathbf{f(U\rel)}-\mathbf{f({U}})-\mathbf{Df({U})}(\mathbf{U\rel}-\mathbf{{U}})\right) \right)^\top\nabla^2_{\mathbf{U}} E^{\alpha}(\mathbf{{U}})\right)\varphi\, dx\, dt\\
&\hspace{3 cm}+ \int_{0}^{T}\int_{I} \left(\mathbf{{b}({U})}\cdot(\mathbf{U\rel}-\mathbf{{U}})\right)^\top\nabla^2_{\mathbf{U}} E^{\alpha}(\mathbf{{U}})\varphi\, dx\, dt\\
&=\int_{0}^{T}\int_{I} \left(\nabla^2_{\mathbf{U}} E^{\alpha}(\mathbf{{U}})\left(\mathbf{f(U\rel)}-\mathbf{f({U}})-\mathbf{Df({U})}(\mathbf{U\rel}-\mathbf{{U}})\right) \cdot \mathbf{{U}}_x\right)\varphi\, dx\, dt\\
&\hspace{3 cm}+ \underbrace{\int_{0}^{T}\int_{I} \nabla^2_{\mathbf{U}} E^{\alpha}(\mathbf{{U}})(\mathbf{U\rel}-\mathbf{{U}}))\cdot \mathbf{{b}({U})}\varphi\, dx\, dt}_{=:\mathcal{I}_3},
\end{aligned}
\end{equation}
where we have used the symmetry of the Hessian matrix. \\
Now in view of \eqref{reformulated_dispersion} and the definition of $E^{\alpha}$, it is easy to obtain that
\begin{equation}\label{taylor_reminder_F}
\begin{aligned}
  \nabla^2_{\mathbf{U}} E^{\alpha}({\mathbf{U}})
  \big(\mathbf{f}(\mathbf{U}^\alpha)-\mathbf{f}({\mathbf{U}})
  &- \mathbf{Df}({\mathbf{U}})(\mathbf{U}^\alpha-{\mathbf{U}})\big)
  \cdot {\mathbf{U}}_x\\
  &= \sgn(\gamma)\big(f(u^\alpha)-f({u})-f'({u})(u^\alpha-{u})\big)
     \big(f'({u}){u}_x+{\psi}_x\big).
\end{aligned}
\end{equation}
Moreover, $\mathcal{I}_2$ from \eqref{difference_idenity_weak_form} and $\mathcal{I}_3$ from \eqref{I_1} together become
\begin{equation}\label{I2+I3}
    \begin{aligned}
\mathcal{I}_2+\mathcal{I}_3&=\int_{0}^{T}\int_{I} \left(\nabla^2_{\mathbf{U}}E^{\alpha}(\mathbf{{U}})(\mathbf{U\rel}-\mathbf{{U}}))\cdot \mathbf{{b}({U})}+\left(\nabla_{\mathbf{U}}E^{\alpha}(\mathbf{U}\rel)-\nabla_{\mathbf{U}}E^{\alpha}(\mathbf{{U}})\right)\cdot \mathbf{b(U\rel)}\right)\varphi\, dx\, dt\\
    &=\int_{0}^{T}\int_{I} \left(\dfrac{|\gamma| }{\alpha} (\psi\rel-{\psi}) {u}_{xxt}-\beta (w\rel-{w}){u}_{tt}\right)\varphi\, dx\, dt.
    \end{aligned}
\end{equation}
Thus, using \eqref{I_1}, \eqref{taylor_reminder_F} and \eqref{I2+I3} in \eqref{difference_idenity_weak_form}, the integral $\mathcal{I}$ in \eqref{difference_idenity_weak_form} reduces to  
\begin{equation}\label{difference_idenity_weak_form_simplified}
\begin{aligned}\mathcal{I}&=    \int_{0}^{T}\int_{I} \left(\mathcal{E}^{\alpha}(\mathbf{U}\rel|\mathbf{{U}})\varphi_t+\mathcal{Q}^{\alpha}(\mathbf{U}\rel|\mathbf{{U}})\varphi_x\right) \, dx\, dt\\
&\qquad +\int_{0}^{T}\int_{I} \sgn(\gamma)\big(f(u^\alpha)-f({u})-f'({u})(u^\alpha-{u})\big)
     \big(f'({u}){u}_x+{\psi}_x\big)\varphi\, dx\, dt\\
&\qquad +\int_{0}^{T}\int_{I}\left(\dfrac{|\gamma| }{\alpha} (\psi\rel-{\psi}) {u}_{xxt}-\beta (w\rel-{w}){u}_{tt}\right)\varphi\, dx\, dt+\int_{I} \mathcal{E}^{\alpha}(\mathbf{U\rel_0}|\mathbf{{U}}_0)\varphi(\cdot,0)\, dx\geq 0. 
\end{aligned}
\end{equation}
Since we are using periodic boundary conditions for the  bounded domain 
$\mathcal{I}$, we  get for  the test function $\varphi(x,t)=\vartheta(t)\in C_{0}^{\lip}([0, T))$ the derivatives  $\vartheta_t=\vartheta'(t)$ and $\vartheta_x=0$. In such a situation, the integral identity \eqref{difference_idenity_weak_form_simplified} simplifies to
the following form
\begin{equation}\label{difference_idenity_periodic}
\begin{aligned}
    &\hspace*{-0.2cm}\int_{0}^{T}\int_{I} \mathcal{E}^{\alpha}(\mathbf{U}\rel| {\mathbf U})\vartheta'\, dx\, dt+\int_{I} \left(\mathcal{E}^{\alpha}(\mathbf{U}_0\rel| {\mathbf U}_0)\right)\vartheta(\cdot, 0)\, dx\\
&\geq -\int_{0}^{T}\int_{I}\sgn(\gamma)\big(f(u^\alpha)-f({u})-f'({u})(u^\alpha-{u})\big)
     \big(f'({u}){u}_x+{\psi}_x\big)\vartheta\, dx\, dt\\
& \quad -\int_{0}^{T}\int_{I}\left(\dfrac{|\gamma| }{\alpha} (\psi\rel-{\psi}) {u}_{xxt}-\beta (w\rel-{w}){u}_{tt}\right)\vartheta \, dx\, dt.
\end{aligned}
\end{equation}
Following the proof of Theorem 5.2.1 in \cite{dafermos2005hyperbolic}, we fix the test function $\vartheta$ for some time $\sigma\in (0, T)$. The test function $\vartheta$ for some $\epsilon>0$ is given by
\begin{align}
  \vartheta(t)
  &= 
  \begin{cases}
    1, & 0 \le t < \sigma,\\[0.3em]
    \dfrac{\sigma - t}{\epsilon} + 1, & \sigma \le t \le \sigma + \epsilon,\\[0.3em]
    0, & t \ge \sigma + \epsilon,
  \end{cases} 
\end{align}
such that 
\begin{align}
  \vartheta'(t)
  &= 
  \begin{cases}
    0, & 0 \le t < \sigma,\\[0.3em]
    -\dfrac{1}{\epsilon}, & \sigma < t < \sigma + \epsilon,\\[0.3em]
    0, & t \ge \sigma + \epsilon.
  \end{cases}
\end{align}
Clearly, \eqref{difference_idenity_periodic} then reduces to 
\begin{equation}\label{difference_idenity_periodic_updated}
\begin{aligned}
    \hspace*{-0.2cm}\dfrac{1}{\epsilon}\int_{\sigma}^{\sigma+\epsilon}\int_{I} \mathcal{E}^{\alpha}(\mathbf{U}\rel| {\mathbf U})\, dx\, dt\leq&\int_{I} \mathcal{E}^{\alpha}(\mathbf{U}_0\rel| {\mathbf U}_0)\, dx+\underbrace{\int_{0}^{\sigma}\int_{I}\mathcal{R}\vartheta(t)\, dx\, dt}_{=:\mathcal{I}_4}+\underbrace{\int_{\sigma}^{\sigma+\epsilon}\int_{I}\mathcal{R}\vartheta(t)\, dx\, dt}_{=:\mathcal{I}_5},
\end{aligned}
\end{equation}
where $\mathcal{R}=\sgn(\gamma)\big(f(u^\alpha)-f({u})-f'({u})(u^\alpha-{u})\big)\big(f'({u}){u}_x+{\psi}_x\big)+\left(\dfrac{|\gamma| }{\alpha} (\psi\rel-{\psi}) {u}_{xxt}+\beta (w\rel-{w}){u}_{tt}\right).$\\
Now for $\mathcal{I}_4$, in view of Young's inequality, we have
\begin{equation}\label{I_4}
\begin{aligned}
    \mathcal{I}_4\leq &\, {\lVert \big(f'({u}){u}_x+{\psi}_x\big)\rVert}_{L^{\infty}( I \times (0, \sigma))} \int_{0}^{\sigma}\int_{I}\left|\big(f(u^\alpha)-f({u})-f'({u})(u^\alpha-{u})\big)\right|\, dx\,dt\\
&+\int_{0}^{\sigma}\int_{I}\left(\dfrac{1}{2\alpha}(\psi^\alpha-{\psi})^2+\dfrac{\beta}{2} (w^\alpha-{w})^2\right)\,dx\,dt+\int_{0}^{\sigma}\int_{I}\left(\dfrac{\gamma^2}{2\alpha}{u}_{xxt}^2+\dfrac{\beta}{2} {u}_{tt}^2 \right)\,dx\,dt.
\end{aligned}
\end{equation}
In view of Assumption \ref{ass_fmonotone}, we have ${\lVert f''\rVert}_{\infty}\leq L$ and thus
\begin{align}\label{eq: taylor_f}
\big(f(u^\alpha)-f({u})-f'({u})(u^\alpha-{u})\big)\leq \dfrac{L}{2}(u\rel-{u})^2.
\end{align}
Moreover, the uniform convexity of  $\sgn(\gamma)F$ with $\sgn(\gamma)f'(u)\geq \delta>0$ ensures
\begin{align}\label{rel_ent_2}
\dfrac{L}{2}(u\rel-{u})^2\leq \dfrac{L\sgn(\gamma)}{2\delta} \big(F(u\rel)-F({u})-f({u})(u\rel-{u})\big).
\end{align}
Therefore, using \eqref{eq: taylor_f} and \eqref{rel_ent_2} in \eqref{I_4} and in view of the relative entropy from \eqref{rel_entropy_hyp_system}, we have
\begin{align*}
    \mathcal{I}_4&\leq C_3\left(L, \delta\right)  {\lVert {u}_{xxx}\rVert}_{L^{\infty}( I \times (0, \sigma) )}\int_{0}^{\sigma}\int_{I}\mathcal{E}^{\alpha}(\mathbf{U\rel}|\mathbf{{U}})\, dx\,dt+\int_{0}^{\sigma}\int_{I}\left(\dfrac{\gamma^2}{2\alpha}{u}_{xxt}^2+\dfrac{\beta}{2} {u}_{tt}^2 \right)\,dx\,dt,
\end{align*}
where the constant $C_3=C\left(L, \delta\right) $ does not depend on the parameters $\alpha$ and $\beta$.\\
Furthermore, for $\mathcal{I}_5$ in \eqref{difference_idenity_periodic_updated}, one can apply the Cauchy-Schwarz inequality to obtain the inequality
\begin{align*}
    \mathcal{I}_5\leq \sqrt{\epsilon |I| }{\lVert\mathcal{R}\rVert}_{{L}^{2}( I \times (\sigma, \sigma+\epsilon))}\rightarrow 0, \quad \text{as}~\epsilon\rightarrow 0.
\end{align*}
Thus, by letting $\epsilon \rightarrow 0$ in \eqref{difference_idenity_periodic_updated}, we obtain using the Sobolev inequality
\begin{align*}
   \hspace{-0.2 cm} \int_{I} \mathcal{E}^{\alpha}(\mathbf{U}\rel
    (x, \sigma)| {\mathbf U}(x, \sigma))\, dx&\leq\int_{I} \mathcal{E}^{\alpha}(\mathbf{U}_0\rel| {\mathbf U}_0)\, dx+\int_{0}^{\sigma}\int_{I}\left(\dfrac{\gamma^2}{2\alpha}{u}_{xxt}^2+\dfrac{\beta}{2} {u}_{tt}^2 \right)\,dx\,dt\\
   & \qquad + C_3\left(L, \delta\right) {\lVert {u}\rVert}_{L^{\infty}((0, \sigma)\times H^4(I))} \int_{0}^{\sigma}\int_{I}\mathcal{E}^{\alpha}(\mathbf{U\rel}|\mathbf{{U})}\, dx\,dt.
\end{align*}
Therefore, arguing as in Theorem 5.2.1 in \cite{dafermos2005hyperbolic} and using Gronwall's inequality, one can conclude that
\begin{equation}
\begin{aligned}
 &\hspace{-0.4 cm}\sup_{t\in (0, T)}  \int_{I}\mathcal{E}^{\alpha}(\mathbf{U}^\alpha|\mathbf{{U}})\, dx\\
 &\leq \exp\left(C_3T{\lVert{u}\rVert}_{L^\infty(0, T;H^4(I))}\right)\left(\int_{I} \mathcal{E}^{\alpha}(\mathbf{U}_0^\alpha|\mathbf{{U}}_0)\, dx+\dfrac{\gamma^2}{2\alpha} {\lVert {u}_{xxt}\rVert}_{{{L}^{2}}(I_T)}^2+\dfrac{\beta}{2} {\lVert{u}_{tt}\rVert}_{{{L}^{2}}(I_T)}^2\right).
\end{aligned}
\end{equation}
From our choice of prepared initial data
$\mathbf{U}\rel_0=\mathbf{{U}}_0$  (see \eqref{IC_BC_Hyp}) we deduce
\begin{align*}
\int_{I}&\mathcal{E}^{\alpha}(\mathbf{U}_0^\alpha|\mathbf{{U}}_0)\, dx\\
&=\sgn(\gamma)\left(F(u_0\rel)-F({u}_0)-f({u}_0)(u_0\rel-{u}_0)\right)+\dfrac{|\gamma|}{2}(p\rel_0-{p}_0)^2+\dfrac{ \beta}{2}  (w\rel_0-{w}_0)^2+\dfrac{1}{2\alpha}(\psi\rel_0-{\psi}_0)^2=0.
\end{align*}
Then if $\alpha\rightarrow \infty$, we get with $\beta = \alpha^{-1}$ the estimate $\int_{I}\mathcal{E}^{\alpha}(\mathbf{U}^\alpha|\mathbf{{U}})\, dx \le C_4  \alpha^{-1}$.  
The convergence rate in \eqref{convergence_rates} is a consequence of \eqref{relative_entropy_relation_main}.
\end{proof}
\begin{remark}
To prove the relative entropy estimates, the condition $f''\in L^\infty(I)$  can be relaxed. What is  really needed is that there exists a constant $K>0$ such that
\[|f(u^\alpha)-f({u})-f'({u})(u^\alpha-{u})|\leq K |u^\alpha-{u}|^2\quad {\rm{uniformly~in}}~I_T.\]
\end{remark}
\subsection{Diffusive-dispersive asymptotics for the hyperbolic-parabolic system \eqref{hyperbolic_parabolic_system}}
In Theorem \ref{global_wellposedness}, we already proved the global well-posedness of the classical solutions for the system \eqref{hyperbolic_parabolic_system}. In this section, we proceed to analyze the convergence of solutions for the initial value problem 
\eqref{hyperbolic_parabolic_system}, \eqref{hyperbolic_parabolic_system_ini} with $\eps >0$ to the solutions of \eqref{eq:diffusivedispersive}, \eqref{eq:diffusivedispersive_ini}. We use again the scaling  $\beta =\alpha^{-1}$, and we employ again 
relative entropy estimates for our second main result in Theorem \ref{theo:dd}. 
Precisely, we prove the following.
\begin{theorem}[Main result II]\label{theo:dd}
  Let   ${u}_0\in H^6(I)$, and suppose that Assumption \ref{ass_fmonotone} holds. We assume  that 
 ${u}\in {L}^{2}(0, T;H^3(I))$ with ${u}_{tt}\in {L}^{2}(I_T)$ and ${\lVert{u}\rVert}_{L^\infty(0, T;H^4(I))}<\infty$
 is a classical solution of \eqref{eq:diffusivedispersive}, \eqref{eq:diffusivedispersive_ini}.
   Consider for $\beta= \alpha^{-1}$, a 
 sequence   ${\{\mathbf{U}^\alpha=(u^\alpha, \psi^\alpha, w^\alpha, p^\alpha)^\top\}}_{\alpha>0}$  
of smooth solutions 
 of \eqref{hyperbolic_parabolic_system}, \eqref{hyperbolic_parabolic_system_ini}.\\
 Then there is a  constant $C= C(L,\delta, {u})\ge 0$ that is independent of $\alpha$ such that
    \begin{align}\label{convergence_ratesdd}
        {\lVert u\rel-{u}\rVert}_{L^{\infty}\big(0, T; {L}^{2}(I)\big)}+ {\lVert p\rel-{u}_x \rVert}_{L^{\infty}\big(0, T; {L}^{2}(I)\big)}\le C\sqrt{\dfrac{1}{\alpha}}
    \end{align}
    holds.
\end{theorem}
\begin{proof} The proof follows the lines of the proof for Theorem \ref{relative_entropy_theorem_1} but exploits the fact that $ u$ and the elements of the sequence ${\{\mathbf{U}^\alpha\}}_{\alpha>0}$  
are smooth functions.\\ 
Define $ \psi = - |\gamma|  u_{xx}$, $ w =  u_t$ and $ p =  u_x$.
Then we deduce from the fact that $ u $ solves the diffusive-dispersive equation \eqref{eq:diffusivedispersive} the validity of the system  
\begin{equation}
\begin{aligned}
{u}_t+f({u})_x+\sgn(\gamma){\psi}_x&=\varepsilon {u}_{xx},\\
    \dfrac{1}{\alpha}{\psi}_t+f({u})_x+\sgn(\gamma){\psi}_x&=-{w}+\dfrac{\varepsilon}{\alpha} {\psi}_{xx}-\dfrac{1}{\alpha} |\gamma| {u}_{xxt},\\
    \beta {w}_t-|\gamma| {p}_x&={\psi}+\beta {u}_{tt},\\
    {p}_t-{w}_x&=\varepsilon  p_{xx}.
\end{aligned}\label{eq_approxdd}
\end{equation}
Substracting \eqref{eq_approxdd}  from   \eqref{hyperbolic_parabolic_system} we obtain the error system 
\begin{equation}
\begin{aligned}
     (u\rel-{u})_t+(f(u\rel)-f({u}))_x+\sgn(\gamma) (\psi\rel-{\psi})_x&=\varepsilon (u\rel-{u})_{xx},\\
    \dfrac{1}{\alpha}(\psi\rel-{\psi})_t+(f(u\rel)-f({u}))_x+\sgn(\gamma)(\psi\rel-{\psi})_x&=-(w\rel-{w})+\dfrac{\varepsilon}{\alpha}(\psi\rel-{\psi})_{xx}+\dfrac{1}{\alpha} |\gamma| {u}_{xxt},\\
    \beta (w\rel-{w})_t-|\gamma| (p\rel-{p})_x&=(\psi\rel-{\psi})-\beta {u}_{tt},\\
    (p\rel-{p})_t-(w-{w})_x&=\varepsilon(p-{p})_{xx}.
\end{aligned}
\end{equation}
We recall the definition of the relative entropy from \eqref{rel_entropy_hyp_system}. 
A straightforward computation using the lower bound $\delta$ of $\sgn(\gamma)f'$ and the upper bound $L$ of $f''$ from Assumption \ref{ass_fmonotone} yields then
\begin{align*}
    \dfrac{d}{dt}\int_{I} \mathcal{E}^{\alpha}(\mathbf{U}\rel|\mathbf{{U}})\, dx+\varepsilon&\int_{I}\left(\delta (u\rel_x-{u}_x)^2+\dfrac{1}{\alpha}(\psi\rel_x-{\psi}_x)^2+|\gamma|(p\rel_x-{p}_x)^2\right)\, dx\\
    &\leq C_5(\delta,L){\lVert{u}_{t}\rVert}_{\infty}\int_{I}\mathcal{E}^{\alpha}(\mathbf{U}\rel|\mathbf{{U}})\, dx+\dfrac{\gamma^2}{\alpha} {\lVert {u}_{xxt}\rVert}_{{L}^{2}}^2+\beta {\lVert{u}_{tt}\rVert}_{{L}^{2}}^2.
\end{align*}
The constant $C_5$ is independent of  $\alpha$. Then, using Sobolev embedding and Gronwall's inequality, we have
\begin{equation}
\begin{array}{rcl}
 \sup_{t\in (0, T)}  \int_{I}\mathcal{E}^{\alpha}(\mathbf{U}^\alpha|\mathbf{{U}})\, dx
 &\leq &\exp\left(C_5(\delta,L)\left({\lVert{u}\rVert}_{L^\infty(0, T;H^4(I)}\right)T\right)\\[2ex]
 && \qquad \times
 \left(\int_{I} \mathcal{E}^{\alpha}(\mathbf{U}_0^\alpha|\mathbf{{U}}_0)\, dx+\dfrac{\gamma^2}{\alpha} {\lVert {u}_{xxt}\rVert}_{{L}^{2}}^2+\beta {\lVert{u}_{tt}\rVert}_{{L}^{2}}^2\right)
\end{array}
\end{equation}
The estimate \eqref{convergence_ratesdd}  is then a consequence 
of the lower bound \eqref{relative_entropy_relation_main} for
the relative entropy $\mathcal{E}^{\alpha}(\mathbf{U}^\alpha|\mathbf{{U}})$
and the fact that $ \mathcal{E}^{\alpha}(\mathbf{U}_0^\alpha|\mathbf{{U}}_0) =0$
due to the choice of the prepared initial data
\eqref{hyperbolic_parabolic_system_ini}.
\end{proof}
For numerical evidence supporting the statement of Theorem \ref{theo:dd} under the scaling choice $\beta = \alpha^{-1}$ we refer to Section \ref{num:KdV} below.

\section{Numerical results}\label{sec: Numerics}
In this section, we provide a series of numerical examples for our new
approximate systems \eqref{hyperbolic_system} and \eqref{hyperbolic_parabolic_system}. For these hyperbolic and hyperbolic-parabolic systems, reliable 
numerical schemes are available. We show that our chosen approach can 
handle strong nonlinear effects leading to oscillating or shock-wave-like 
behavior.  In particular, we provide test cases for a variety of diffusive-dispersive equations ranging from KdV, Gardner, and modified KdV-Burgers equations. We also provide test cases to capture non-classical shock waves such as dispersive shock waves (DSWs) or undercompressive shock waves with our proposed systems, which is very rare in the literature to the best of our knowledge. The results  confirm 
the convergence results obtained in Section \ref{sec: convergence}.\\
The relaxation approach we propose here applies to the rather general class of PDEs \eqref{main}. Therefore, the numerical results we present cover several cases, for both convex and non-convex flux functions, in the purely dispersive or diffusive-dispersive cases; see also Remark \ref{remark_non_convex}(iii).
\subsection{The numerical scheme}
Our discretization  approach is described using the 
abstract form
\begin{equation}
	\mathbf{U}\rel_t + {\mathbf{f}\prn{\mathbf{U}\rel}}_x = \mathbf{S}(\mathbf{U}\rel)\mathbf{U}\rel + {\mathbf D} \mathbf{U}\rel_{xx}
	\label{eq:QL_system}
\end{equation}
for both approximate systems. For $\gamma\in \mathbb{R}\setminus \{0\}$, we refer to \eqref{hyperbolic_system} for the first-order hyperbolic 
approximation of the dispersive equation \eqref{eq:dispersive} with $\eps =0$ and to \eqref{hyperbolic_parabolic_system} for the second-order hyperbolic-parabolic 
approximation of \eqref{eq:dispersive} with $\eps>0$. \\
The computational domain $I_c = [x_L,x_R]$ is subdivided into $N$ equidistant cells of length $\Delta x = (x_R-x_L)/N$. The $i^{th}$ cell is identified by its center $x_i = x_L + \prn{i+1/2} \Delta x$ and is delimited by the boundaries $x_{i\pm\hlf}$. The time $t\in[0,T]$ is discretized so that $t^n = n\, \Delta t$, where $\Delta t$ is subject to a CFL-like condition. We refer to the cell average of the numerical solution in cell $i$, at time $t^n$ by $\mathbf{U}_i^n$. \newline
Let us point out that there are by now multiple well established  numerical methods for \eqref{eq:QL_system}, see, e.g.,~\cite{GR21} for some recent retreat. We propose a semi-implicit scheme based on an explicit second-order two-step Lax-Wendroff method for the hyperbolic first-order operator of the system, combined  with an implicit treatment of the stiff source term and a standard central discretization for the second-order derivatives. The choice of a Lax-Wendroff scheme is justified by its low numerical dissipation, especially since the considered system admits potentially fast characteristic speeds. An explicit diffusive scheme may require an excessively refined mesh to conserve a good quality of the numerical solution at large times.    
An operator splitting approach for the parabolic terms is employed for the parabolic terms for the test cases where they are present.
In this setting, at each time step and in every computational cell, the new solution is computed by first evaluating the intercell values $\mathbf{U}_{i\pm\hlf}^{n+\hlf}$ via 
\begin{subequations}
\begin{align}
	\mathbf{U}_{i+\hlf}^{n+\hlf} = \hlf \prn{\mathbf{U}_{i}^n + \mathbf{U}_{i+1}^n} - \frac{\Delta t}{2 \Delta x} \prn{\mathbf{f} \prn{\mathbf{U}_{i+1}^n} - \mathbf{f} \prn{\mathbf{U}_{i}^n} } + \frac{\Delta t}{2} \, \mathbf{S}\mathbf{U}_{i+\hlf}^{n+\hlf}, \\
	\mathbf{U}_{i-\hlf}^{n+\hlf} = \hlf \prn{\mathbf{U}_{i}^n + \mathbf{U}_{i-1}^n} - \frac{\Delta t}{2 \Delta x} \prn{\mathbf{f} \prn{\mathbf{U}_{i}^n} - \mathbf{f} \prn{\mathbf{U}_{i-1}^n} } + \frac{\Delta t}{2}\, \mathbf{S}\mathbf{U}_{i-\hlf}^{n+\hlf},
\end{align} 
	\label{eq:LW_step1}
\end{subequations}
whose values can be computed explicitly as
\begin{equation}
		\mathbf{U}_{i\pm\hlf}^{n+\hlf} = \prn{\I - \frac{\dt}{2}\mathbf{S}}^{-1}\prn{\hlf \prn{\mathbf{U}_{i}^n + \mathbf{U}_{i\pm1}^n} \mp \frac{\Delta t}{2 \Delta x} \prn{\mathbf{f} \prn{\mathbf{U}_{i\pm1}^n} - \mathbf{f} \prn{\mathbf{U}_{i}^n} }}.
\end{equation} 
The solution of the hyperbolic first-order part of the system is then given by
\begin{align*}
	\mathbf{U}_{i}^{\star} = \mathbf{U}_{i}^n - \frac{\dt}{\dx} \prn{\F\prn{\mathbf{U}_{i+\hlf}^{n+\hlf}} - \F\prn{\mathbf{U}_{i-\hlf}^{n+\hlf}}} + \frac{\dt}{2} \,\mathbf{S} \prn{\mathbf{U}_{i+\hlf}^{n+\hlf} + \mathbf{U}_{i-\hlf}^{n+\hlf}}.
\end{align*} 
Finally, the solution at the new time, accounting for the second-order terms (present for $\eps >0$)  is 
\begin{equation*}
    \mathbf{U}_{i}^{n+1} = \mathbf{U}_{i}^{\star} + \frac{\dt\phantom{^2}}{\dx^2} \mathbf{D}\prn{\mathbf{U}_{i+1}^{\star} -2 \mathbf{U}_{i}^{\star} + \mathbf{U}_{i-1}^{\star}}.
\end{equation*}
The time-step is restricted by a CFL-type condition of the form 
\begin{equation*}
\dt  = \mathrm{CFL} \frac{\dx}{\lambda + \frac{2\varepsilon}{\dx}}, \quad \text{where} \quad \mathrm{CFL} \le 1 \quad \text{and} \quad \lambda = \max_{i\in \{1,\ldots,N\}}\max_{k = 1,\ldots,4} \abs{\lambda_k(\mathbf{U}_i^n)},
\end{equation*}
where $\lambda$ is the maximum signal speed and  CFL is the Courant-Friedrichs-Levy number, which we take equal to $0.9$ in all the carried out simulations.

\subsection{Numerical results}
 In an attempt to help the reader navigate this section, we recall at the beginning of each test case which  evolution equation we are approximating and for which initial conditions. We shall refer to the latter as the \textit{Original Initial Value Problem} (OIVP). The numerical computations are then carried out by using the previously proposed scheme to solve the \textit{Approximate Initial Value Problem} (AIVP) consisting of the PDE \eqref{hyperbolic_system} (or \eqref{hyperbolic_parabolic_system}) and for which the well-prepared initial conditions are supplied as follows
\begin{equation*}
u(x,0) = u_0(x), \quad \psi(x,0) = -|\gamma| u_{0,xx}, \quad w(x,0) = -f(u_0)_x + |\gamma| u_{0,xxx},\quad p(x,0) = u_{0,x}.
\end{equation*}
Note that $u_0$ is the equivalent initial condition function associated with the original  problem \eqref{main}. Concerning boundary conditions, we use either periodic boundary conditions or mixed Neumann-Dirichlet boundary conditions. Although the former is rather straightforward, the latter is designed to mimic homogeneous Neumann boundary conditions $u_x(x_L,t)=u_x(x_R,t) =0$. We shall refer to this as \textit{pseudo-Neumann boundary conditions}. Since the unknown  $c$ relaxes towards $u$ (see \eqref{eq:waveform} and  \eqref{eq:damped_wave_system}), we impose the same boundary conditions for both.  In terms of the evolved variables $u,\psi,w,p$ this translates to  
\begin{subequations}
\begin{alignat}{5}
&u_x(x_L,t)&&=u_x(x_R,t) =0, \quad \psi_x(x_L,t&&)&&=\psi_x(x_R,t) &&=0, \\
&w_x(x_L,t)&&=w_x(x_R,t) =0, \quad  p(x_L,t&&)&&=p(x_R,t) &&=0. 
\end{alignat}
\label{eq:PNBC}
\end{subequations}
\subsubsection{One-Soliton solution (KdV equation)}
\label{test:1soliton}
In this example, the  OIVP is given by 
\begin{equation}
\begin{cases}
u_t + 6uu_x  = \gamma u_{xxx}, \\
u(x,0) = u_0(x) = \frac{V}{2} \mathrm{sech}^2\prn{\sqrt{\frac{V}{4|\gamma|}} \, x}. 
    \end{cases}
    \label{eq:IVP_1sol}
\end{equation}
The exact solution to \eqref{eq:IVP_1sol} is a traveling wave, {i.e.}, $u(x,t) = u_0(x-Vt)$. The computational domain is $I_c=[-2,2]$ with $N=1000$ cells. We take $\gamma=-10^{-2}$. The values of the relaxation parameters are $\alpha =10^3$ and $\beta = 10^{-6}$. We take $V=1$ and the final time is $T=100$. Here, we use periodic boundary conditions so that the final time corresponds to $25$ periods. The periodization of the solution naturally introduces negligible errors at the boundaries since the solution is exponentially decaying but not perfectly flat. The numerical results are provided in Figure \ref{fig:1_sol}.
\begin{figure}[h!]
    \centering
    \begin{subfigure}{0.49\textwidth}
        \centering
        \includegraphics[width=\linewidth]{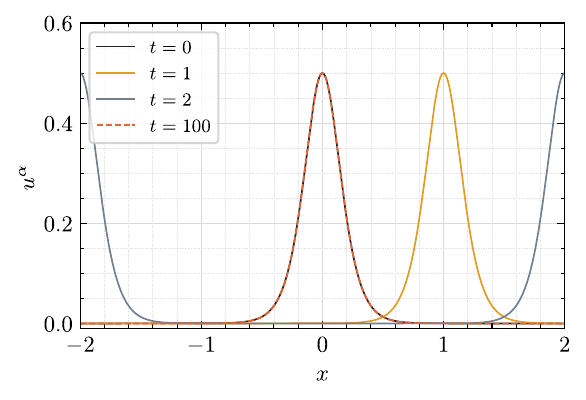}
    \end{subfigure}
    \begin{subfigure}{0.49\textwidth}
        \centering
        \includegraphics[width=\linewidth]{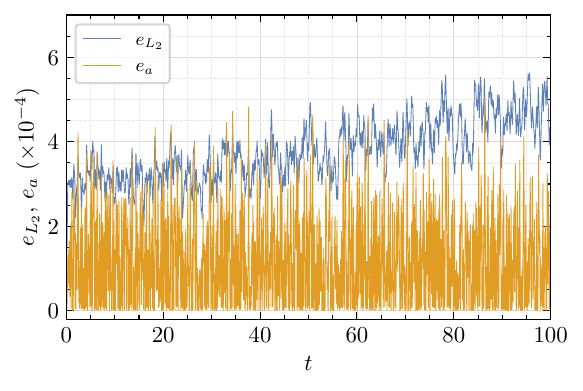}
    \end{subfigure}
\caption{Numerical results for the AIVP corresponding to \eqref{eq:IVP_1sol}. Left: The numerical solution $u(x)$ at different times up to the final time $T=100$. Right: The time evolution of the amplitude error $e_a$ as well as the absolute $L^2$ error between the numerical solution and the exact solution to the OIVP \eqref{eq:IVP_1sol}. $I_c=[-2,2]$ with $N=1000$ cells, $\gamma=-10^{-2}, \alpha =10^3$ and $\beta = 10^{-6}$. }
\label{fig:1_sol}
\end{figure}
The results show that the soliton remains stable for long times and the soliton does not break, not even after $25$ periods.  The numerical solution seemingly behaves as a traveling wave, with the same speed and profile as the solution to \eqref{eq:IVP_1sol}. The long-time behavior of the numerical solution is quantitatively illustrated in the right part of Figure \ref{fig:1_sol}, where the amplitude error $e_a(t)$ and the absolute 
$L^2$-error between the numerical solution of the AIVP and the exact solution of the OIVP $e_{L^2}(t)$, given by
\begin{equation*}
e_a(t^n) = \abs{\max_i(u_i^n)-V/2}, \quad e_{L^2}(t^n) = \sqrt{\sum_i \prn{u_i^n - u_0(x_i-Vt^n)}^2},
\end{equation*}
are plotted over time. The evolution of the error over the time  interval $[0,100]$ shows no significant increase in the order of magnitude.\\
This shows the validity of our numerical approach even for large times. We conjecture that the long-time dynamics of the solution of the approximate system \eqref{hyperbolic_system}  is also governed by a stable soliton solution under small perturbations, due to both periodization and relaxation. Thus, the approximation   
resembles qualitative properties of the original  equation \eqref{eq:dispersive}.
\subsubsection{Two-Soliton solution (KdV equation)}
\label{test:2soliton}
To address a more complex example, we consider  an  OIVP given by 
\begin{equation}
\begin{cases}
u_t + 6uu_x  = - u_{xxx} , \\
u(x,0) =
\frac{
  2 (V_1 - V_2)
  \prn{
    V_1 \cosh^{2}\theta_2(x)
    + V_2 \sinh^{2}\theta_1(x)
  }
}{
  \prn{
    (\sqrt{V_1} + \sqrt{V_2}) 
    \cosh\left(\theta_1(x) - \theta_2(x)\right)
    + (\sqrt{V_1} - \sqrt{V_2}) 
    \cosh\left(\theta_1(x) + \theta_2(x)\right)
  }^2
},
    \end{cases}
    \label{eq:IVP_2sol}
\end{equation}
where $\theta_j(x) = \sqrt{V_j}(x-x_j)/2,\ j=1,2$.
The computational domain is $I_c=[-15,15]$ with $N=1000$ cells. The values of the relaxation parameters are $\alpha =4\times10^3$ and $\beta = 10^{-6}$. Periodic boundary conditions are used and we set the final time to $T=60$. The solution consists of two solitons that move with different velocities and pass each other. Upon interacting, a phase shift occurs before each soliton resumes its course. Different aspects of the numerical solution are provided in Figure \ref{fig:2-Soliton}.\\
Again, we observe that our numerical approach provides  valid results for large times. 
We point out that the approximate system \eqref{hyperbolic_system} allows for crossings of solitons just like the dispersive target problem \eqref{eq:dispersive}.
\begin{figure}[h!]
    \centering
    \begin{subfigure}{0.49\textwidth}
        \centering
        \includegraphics[width=\linewidth]{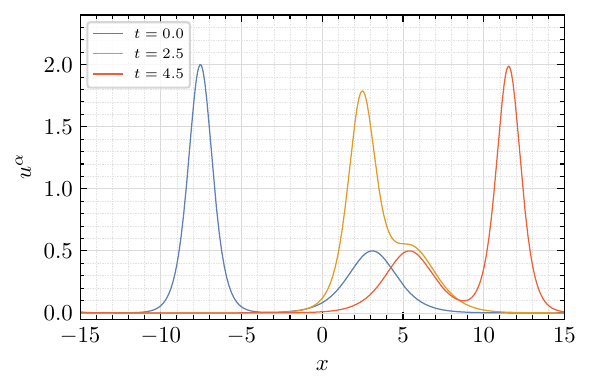}
        \caption{Numerical solution $u\rel$ for $t\in\{0,2.5,4.5\}$}
        \label{fig:2-soliton_a}
    \end{subfigure}
    \begin{subfigure}{0.49\textwidth}
        \centering
        \includegraphics[width=\linewidth]{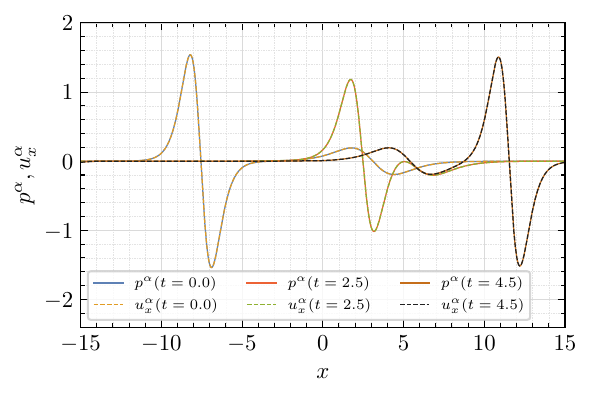}
        \caption{Comparison of $p^\alpha$ and $u_x\rel$ for $t\in\{0,2.5,4.5\}$}
        \label{fig:2-soliton_b}
    \end{subfigure}

        \begin{subfigure}{0.49\textwidth}
        \centering
        \includegraphics[width=1.1\linewidth]{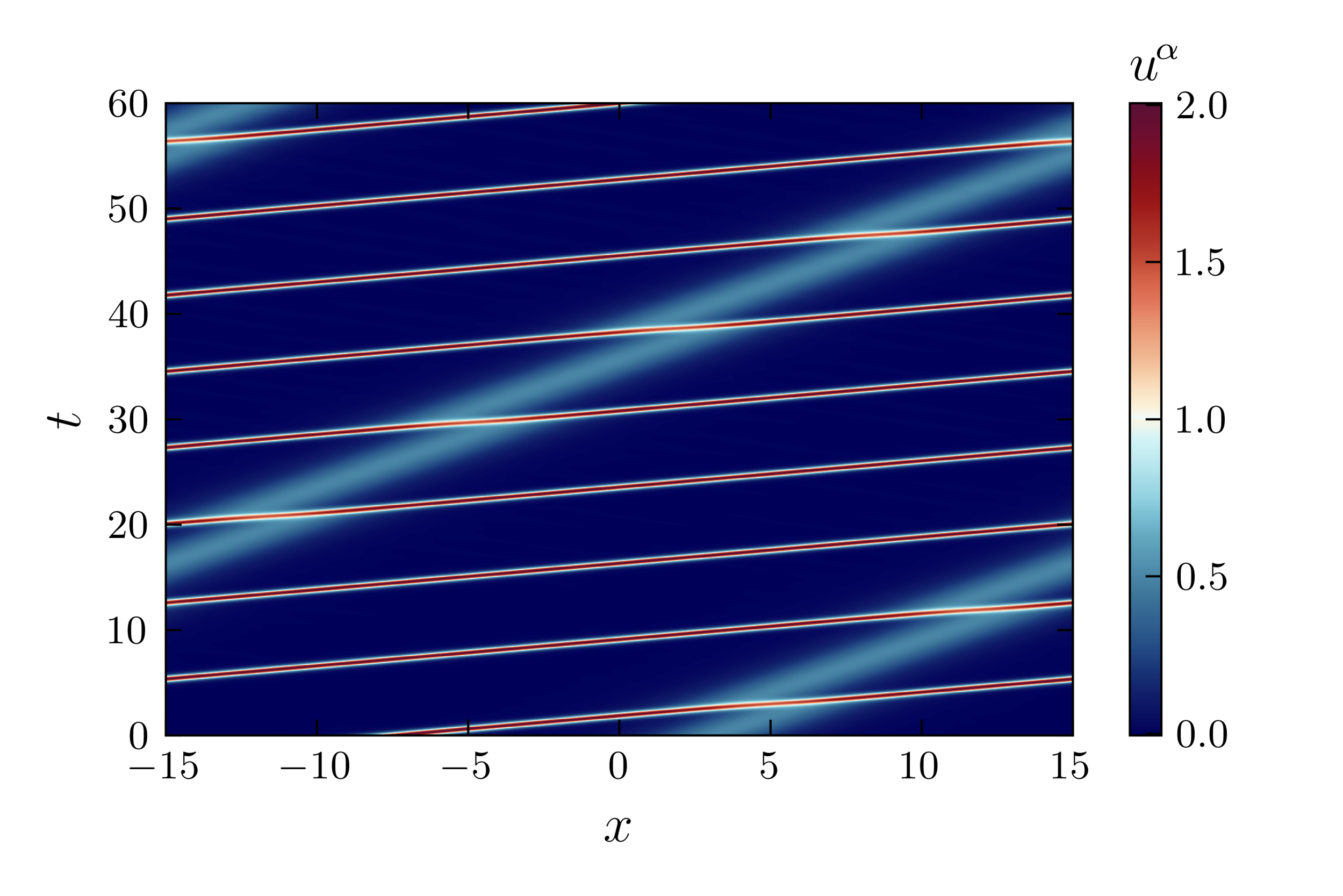}
        \caption{$x-t$ diagram of $u\rel$ for $t\in[0,60]$}
        \label{fig:2-soliton_c}
    \end{subfigure}
        \begin{subfigure}{0.49\textwidth}
        \centering
        \includegraphics[width=1.1\linewidth]{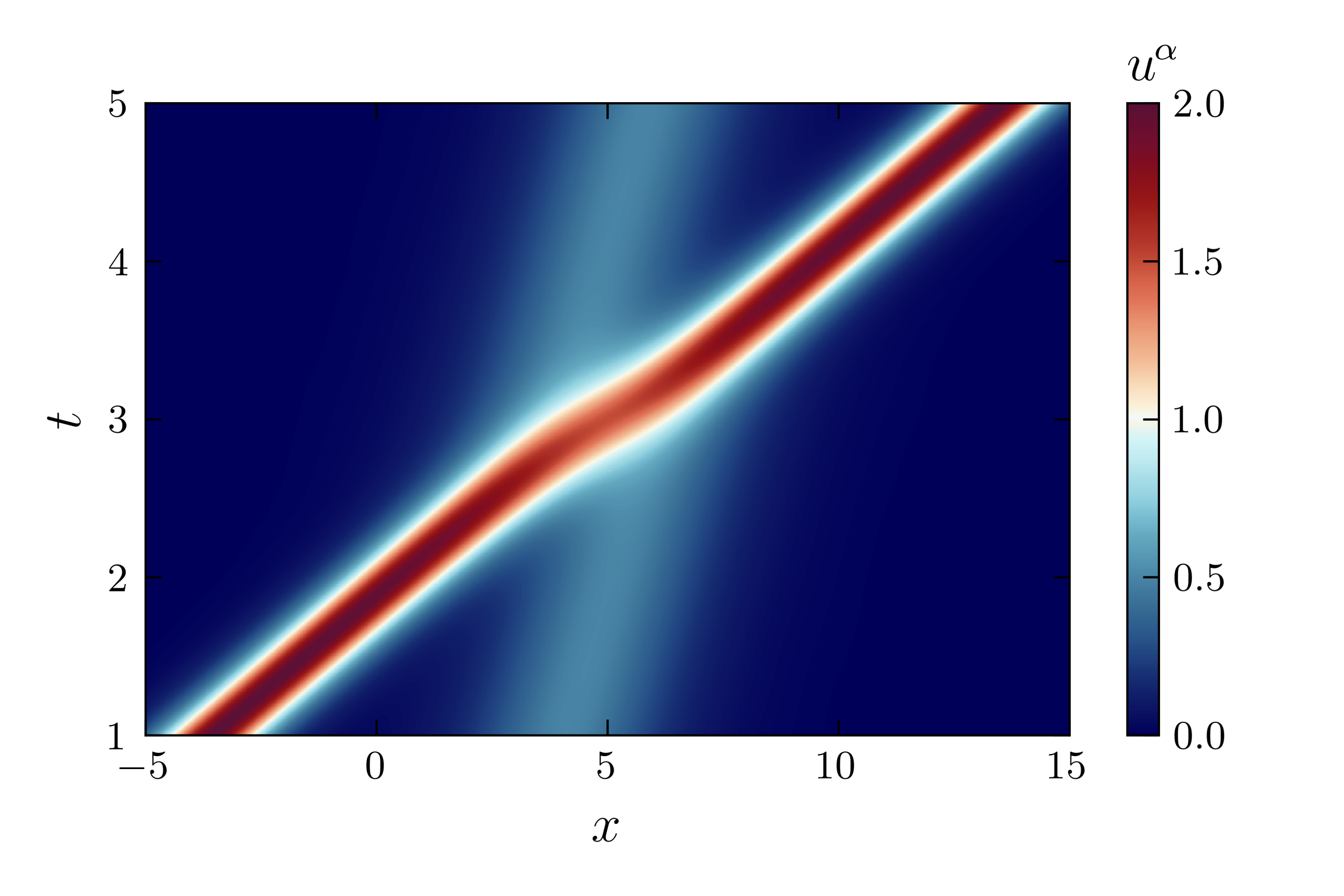}
        \caption{$x-t$ diagram of $u\rel$, zoom on the interaction}
        \label{fig:2-soliton_d}
    \end{subfigure}
\caption{Numerical results for the AIVP corresponding to \eqref{eq:IVP_2sol}, showing the main dynamics of the two-soliton solution. The parameters are $I_c=[-15,15]$ with $N=1000$ cells. $\gamma=-1$, $\alpha =4\times10^3$, $\beta = 10^{-6}$. $T=60$ and boundary conditions are periodic.}
\label{fig:2-Soliton}
\end{figure}
In Figure \ref{fig:2-Soliton_E}, we show the evolution over time of the energy error  with respect to its initial value, that is $E\rel-E_0\rel$. We deliberately do not use absolute value, only to emphasize that the total energy computed from the numerical solution, increases from its initial value and does not show monotonicity, given that the used numerical scheme is not entropy stable. Nevertheless, we report in Table  \ref{tab:conv_E2sol} the values of $\nrm{E\rel-E_0\rel}_{L^\infty(0,T)}$, corresponding to the amplitude of the periodic peaks, and which show second order convergence rate.
\begin{figure}[h!]
    \centering
    \includegraphics[width=0.5\linewidth]{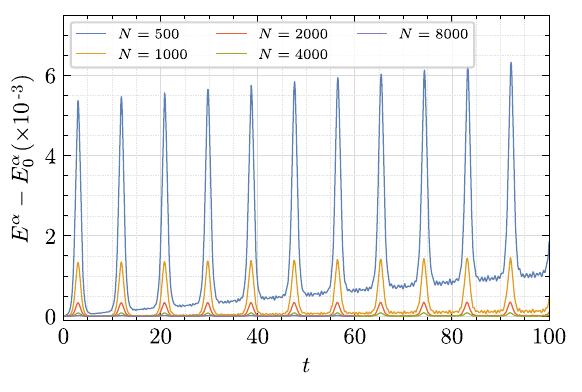}
    \caption{Time evolution of the total energy error with respect to its initial value $E\rel-E\rel_0$ for different mesh sizes. The parameters are $I_c=[-15,15]$, $\gamma=-1$, $\alpha =2\times10^3$, $\beta = 10^{-8}$, $T=100$, and boundary conditions are periodic.}
    \label{fig:2-Soliton_E}
\end{figure}
\begin{table}[h!]
\centering
\begin{tabular}{|c|c|c|}
\hline
$N$ & $\nrm{E\rel-E_0\rel}_{L^\infty(0,T)}$ & Order \\
\hline
$500$  & $6.32\times 10^{-3}$ & --   \\
$1000$ & $1.46\times 10^{-3}$ & 2.11 \\
$2000$ & $3.50\times 10^{-4}$ & 2.06 \\
$4000$ & $8.57\times 10^{-5}$ & 2.02 \\
$8000$ & $2.10\times 10^{-5}$ & 2.02 \\
\hline
\end{tabular}
\caption{Convergence of $\nrm{E\rel-E_0\rel}_{L^\infty(0,T)}$, corresponding to the curves shown in Figure \ref{fig:2-Soliton_E}.}
\label{tab:conv_E2sol}
\end{table}

\subsubsection{Negative hump DSW (KdV equation)}
\label{test:sech_DSW}
The OIVP is given by 
\begin{equation}
\begin{cases}
u_t + 6uu_x=  \gamma u_{xxx} , \\
u(x,0) = -\mathrm{sech}^2(x). 
    \end{cases}
    \label{eq:IVP_dsw_sech}
\end{equation}
This test case is adapted from \cite{grava2007numerical}, where the
initial value problem \eqref{eq:IVP_dsw_sech} is analyzed. The computational domain is $I_c=[-5,5]$ with $N=8000$ cells. We take $\gamma=-10^{-4}$. The values of the approximation parameters are $\alpha =10^3$ and $\beta = 10^{-7}$. The final time is $T=0.4$. The solution dynamics are such that the smooth initial datum evolves up to a given finite time, where it develops a wave-train of rapid oscillations in a region of space expanding over time. This is usually referred to as a Collisionless  \cite{gurevich1973nonstationary} or Dispersive Shock Wave (DSW). The main dynamics of the numerical solution are shown in Figure \ref{fig:DSW_sech}. In particular, in Figure \ref{fig:sech-a}, we show the numerical solution at different times $t\in\{0,0.1,t_c,0.4\}$ where $t_c=0.216$ is the predicted time marking the onset of the DSW, see~\cite{grava2007numerical}. In Figure \ref{fig:sech-b}, we compare $p^\alpha$ and $u^\alpha_x$ in the DSW region at the final time $t=0.4$. Then, we also compare the numerical solution at the latter time with the corresponding Whitham envelope for the solution to the OIVP \eqref{eq:IVP_dsw_sech} in Figure \ref{fig:sech-c}, which shows excellent agreement. Finally, the space-time evolution of the DSW region is summarized in the $x-t$ diagram of Figure \ref{fig:sech-d}, superimposed with the asymptotic $(x,t)$ coordinates of the DSW edges. All these results show an excellent performance of the method, since the resulting numerical solution to the AIVP is seemingly indistinguishable from the exact solution of the  OIVP \eqref{eq:IVP_dsw_sech}, in terms of space-time evolution but also regarding the dispersion effects in a test case where they are not negligible.        
\begin{figure}[!h]
    \centering
    \begin{subfigure}{0.5\textwidth}
        \centering
        \includegraphics[width=\linewidth]{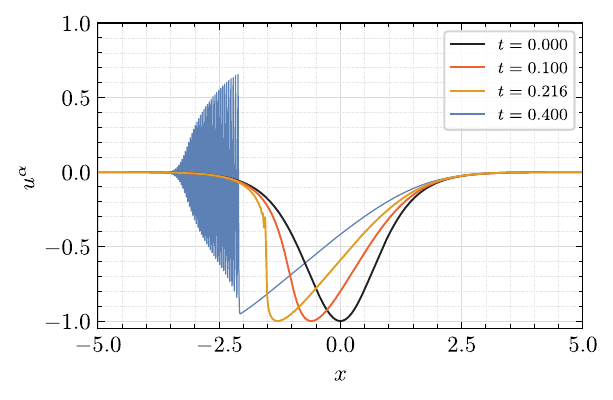}
        \caption{Numerical solution $u$ at $t\in\{0,0.1,0.216,0.4\}$.}
                    \label{fig:sech-a} 
    \end{subfigure}%
    \begin{subfigure}{0.5\textwidth}
        \centering
        \includegraphics[width=\linewidth]{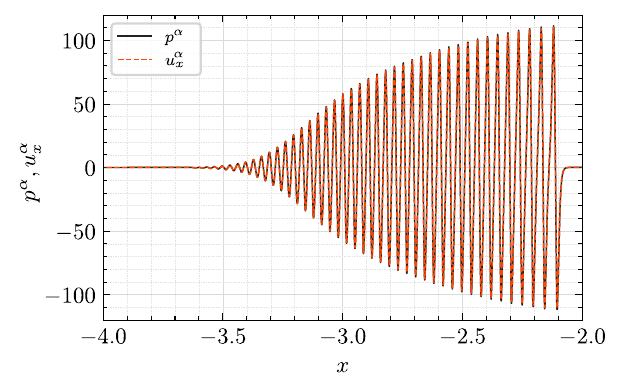}
                \caption{$p^\alpha$ and $u_x$ in the DSW region at $t=0.4$.}
                \label{fig:sech-b}
    \end{subfigure}

    \begin{subfigure}{0.5\textwidth}
        \centering
        \includegraphics[width=\linewidth]{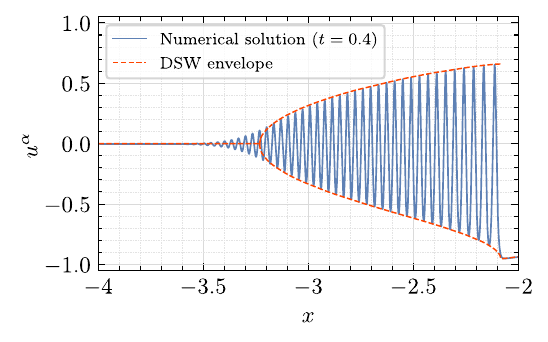}
                \caption{Comparison of the numerical solution with \\the Whitham envelope for the OIVP \eqref{eq:IVP_dsw_sech}.}
                \label{fig:sech-c}
    \end{subfigure}%
    \begin{subfigure}{0.5\textwidth}
        \centering
        \includegraphics[width=1.1\linewidth]{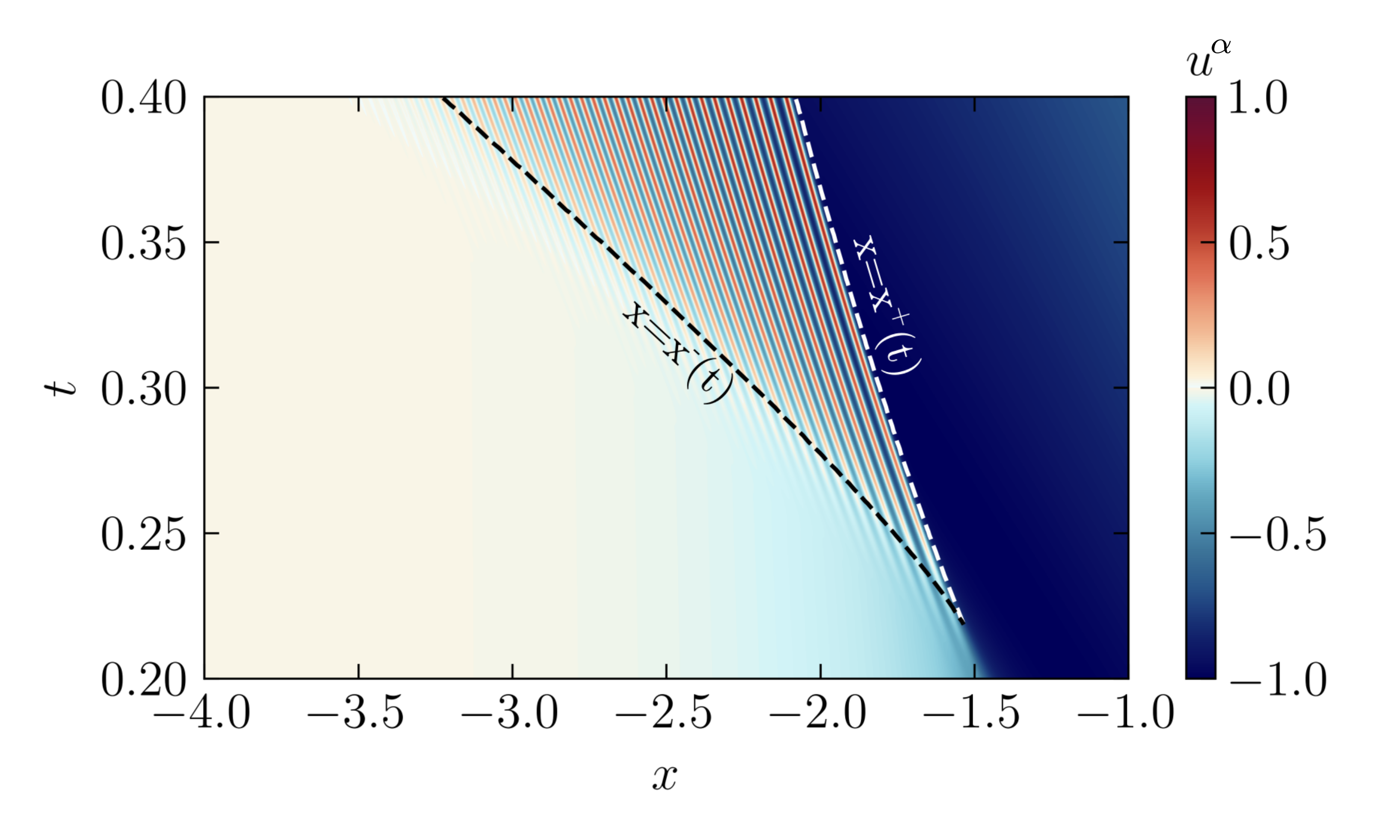}
                \caption{$x-t$ diagram of the numerical solution at the onset of the DSW formation.}
                \label{fig:sech-d}
    \end{subfigure}
\caption{Numerical results for the AIVP corresponding to \eqref{eq:IVP_dsw_sech}. The parameters are $\gamma=-10^{-4}$, $\alpha=10^3$, $\beta=10^{-7}$. The computational domain is $I_c=[-5,5]$ with $N=8000$ cells, and the final time is $T=0.4$. Periodic boundary conditions are used. }
    \label{fig:DSW_sech}
\end{figure}
\subsubsection{Riemann Problem DSW (KdV equation)}
\label{test:DSW}
The OIVP is given by 
\begin{equation}
\begin{cases}
u_t + uu_x  = \gamma u_{xxx}, \\
u(x,0) = \begin{cases}
            1 \quad \text{if} \quad x <0, \\
            0 \quad \text{if} \quad x\ge 0. 
         \end{cases} 
    \end{cases}
    \label{eq:IVP_RPDSW}
\end{equation} 
The computational domain is $I_c=[-8,2]$ with $N=10000$ cells. We take $\gamma=-10^{-4}$. The values of the relaxation parameters are $\alpha =10^3$ and $\beta = 10^{-7}$. The final time is $T=3$. With $\gamma\ne0$, the initial discontinuity develops a DSW (or undular bore). The derivative of the initial data in the sense of distributions is a minus Dirac delta function, which does not belong to ${L}^{2}(I)$. The total energy~\eqref{convex-entropy} would be formally infinite should \(p\) be initialized as \(p = u_{0,x}\) for the AIVP. Instead, we consider an analytic approximation of the initial step function given by 
\begin{equation*}
\tilde{u}_0(x) = \hlf \prn{1-\mathrm{tanh}\prn{x/\omega}}, \qquad \omega=10^{-3}.
\end{equation*}
The asymptotic solution to the OIVP  (for $t\to+\infty$) was studied in \cite{gurevich1973nonstationary} and is given by 
\begin{equation}
u(x,t)
 = 2\,\mathrm{dn}^2\!\left(
   \frac{1}{\sqrt{6}} 
   \left( x - \frac{1 + s^2}{3} t \right),
   s
 \right)
 - (1 - s^2),
 \label{eq:elliptic_sol}
\end{equation}
where $\mathrm{dn}$ is the Jacobi delta amplitude function and $s$ is the elliptic modulus, and which obeys the implicit equation
\begin{equation}
\frac{1 + s^2}{3}
 - \frac{2}{3}\,
   \frac{s^2(1 - s^2)\mathrm{K}(s)}
        {\mathrm{E}(s) - (1 - s^2)\mathrm{K}(s)}
 = \frac{x}{t},
 \label{eq:DSW_tau}
\end{equation}
with $\mathrm{K}$ and $\mathrm{E}$, the complete elliptic integrals of first and second kind, respectively. The asymptotic bounds of the DSW region are then determined by taking the limits $s\to0$ and $s\to1$ in \eqref{eq:DSW_tau} and which yields \cite{gurevich1973nonstationary}
\begin{equation*}
\frac{x}{t} = \tau^- = -1, \quad \frac{x}{t} = \tau^+ = \frac{2}{3}.
\end{equation*}
Upon computing $s$ as a function of $x/t$ in the DSW region, the local extrema of \eqref{eq:elliptic_sol} define the upper and lower asymptotic envelopes $\mathrm{A}_{\pm}(x/t)$ defined as
\begin{equation}
\mathrm{A}_\pm(x/t) = 1 \pm s^2(x/t).
\end{equation}
The numerical results are given in Figure \ref{fig:RPDSW}. The comparison shows that both the amplitude of the oscillations as well as their position and velocity match well with the asymptotic values of the original equation, confirming once again the quality of the approximation. Given the structure of the solution, the vanishing oscillations propagating from the trailing edge of the DSW to the leftmost boundary are usually a source of concern from the numerical point of view especially when the boundary conditions are not completely transparent \cite{besse2022perfectly}. 
\begin{figure}[!h]
    \centering
    \begin{subfigure}{0.5\textwidth}
        \centering
        \includegraphics[width=\linewidth]{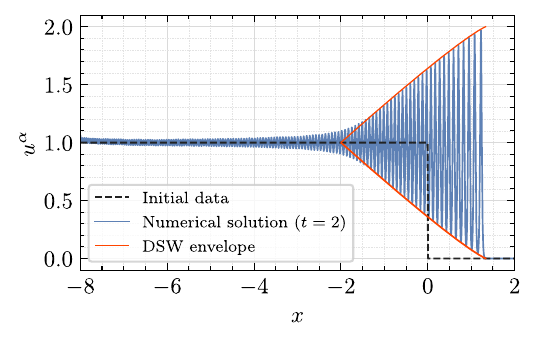} 
    \end{subfigure}%
    \begin{subfigure}{0.5\textwidth}
        \centering
        \includegraphics[width=1.12\linewidth]{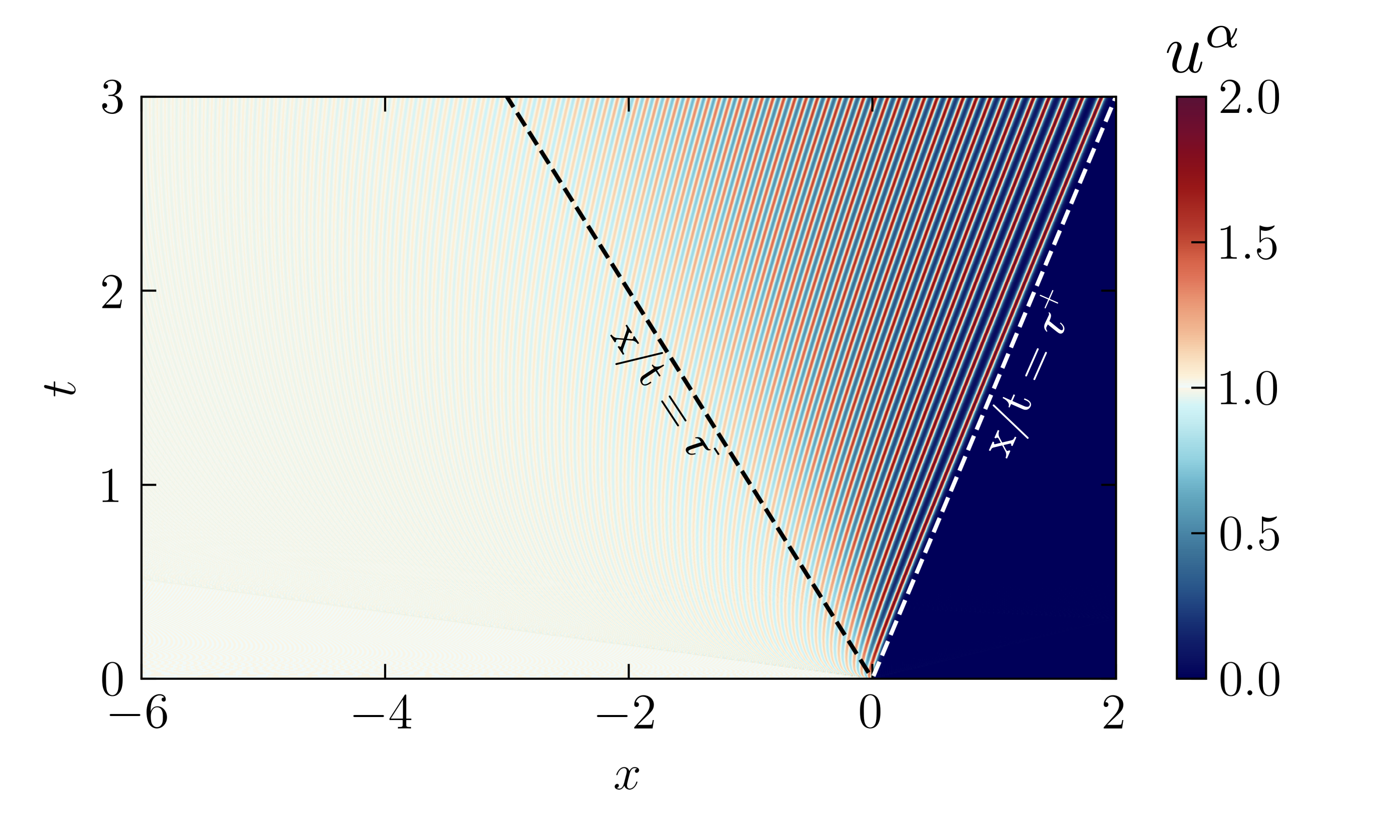}
    \end{subfigure}
\caption{Numerical results for the AIVP corresponding to \eqref{eq:IVP_dsw_sech}. Left: Numerical solution at $t=2$ compared with the asymptotic envelopes $\mathrm{A}_\pm$. Right: $x-t$ diagram of the numerical solution, superimposed with the asymptotic DSW bounds. The parameters are $\gamma=-10^{-4}$, $\alpha=10^3$, $\beta=10^{-7}$. The computational domain is $I_c=[-8,2]$ with $N=10000$ cells, and the simulation is run up to $T=3$. Pseudo Neumann boundary conditions are used.}
    \label{fig:RPDSW}
\end{figure}
Nevertheless, the numerical solution seems to perform well in that regard, even by simply implementing pseudo-Neumann boundary conditions. To further highlight the relevancy of this choice, we compare this same test case using pseudo-Neumann boundary conditions against classic Neumann boundary conditions, for which $p_x(x_L,t) = p_x(x_R,t) =0$. The comparison is reported on Figure \ref{fig:RPDSW_BC} and shows that classic Neumann boundary conditions fail to maintain the correct structure of the solution near the boundary. The consequences of this erroneous choice are visibly more significant on the variable $u\rel$ than on $p\rel$ itself.  
\begin{figure}[!h]
    \centering
    \begin{subfigure}{0.5\textwidth}
        \centering
        \includegraphics[width=\linewidth]{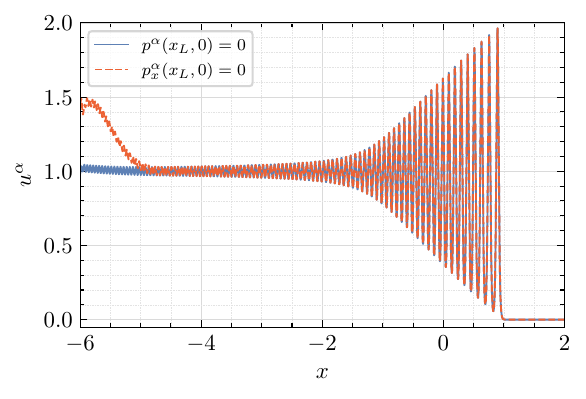} 
    \end{subfigure}%
    \begin{subfigure}{0.5\textwidth}
        \centering
        \includegraphics[width=1.05\linewidth]{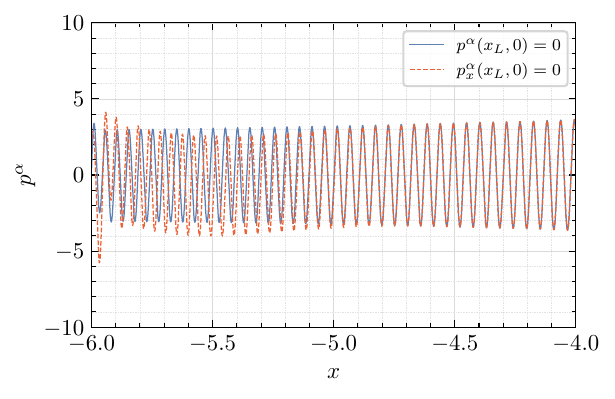}
    \end{subfigure}
\caption{Comparison of pseudo Neumann boundary conditions \eqref{eq:PNBC} with classic Neumann boundary conditions for which $p_x(x_L,t) = p_x(x_R,t) =0$, for the DSW test case. The computational domain is $I_c=[-6,2]$ with $N=10000$ cells. The time is $T=1.5$. The rest of the parameters are $\alpha=10^3$, $\beta=10^{-8}$, $\gamma=-10^{-4}$.}
    \label{fig:RPDSW_BC}
\end{figure}

\subsubsection{Traveling wave front of KdV-Burgers\label{num:KdV}}
The OIVP is given by 
\begin{equation}
\begin{cases}
u_t + uu_x  = \varepsilon u_{xx}  + \gamma u_{xxx} , \\
u(x,0) = u_0(x) = \frac{-3\varepsilon^2}{25
   \gamma }\prn{2 + 2 \tanh \prn{\frac{\varepsilon x}{10\gamma} }+ \mathrm{sech}^2\prn{\frac{\varepsilon x}{10\gamma} }}.
    \end{cases}
    \label{eq:IVP_sechtanh}
\end{equation}
The exact solution to this IVP is a traveling wave $u_0(x-Vt)$ where $V=-\frac{6 \varepsilon^2}{25\gamma}$. We take for this test case $I_c = [-3,1]$, discretized over $N=2000$ computational cells. We take $\gamma=10^{-4}$, $\varepsilon=10^{-2}$ and the values of the relaxation parameters are $\alpha=2\times10^3$ and $\beta=10^{-6}$.
The final time is $T=10$.
Pseudo Neumann boundary conditions \eqref{eq:PNBC} are used. 
The numerical results are given in Figure \ref{fig:sechtanh}. In particular, we show the numerical solution $u\rel$ at different times as well as the time evolution of $E\rel$. To validate the decay rate of the solution, we compared it with the reference decay law \eqref{parabolic_estimate}, integrated over the computational domain and in time using an explicit Euler scheme. 
\begin{figure}[!h]
    \centering
    \begin{subfigure}{0.5\textwidth}
        \centering
        \includegraphics[width=\linewidth]{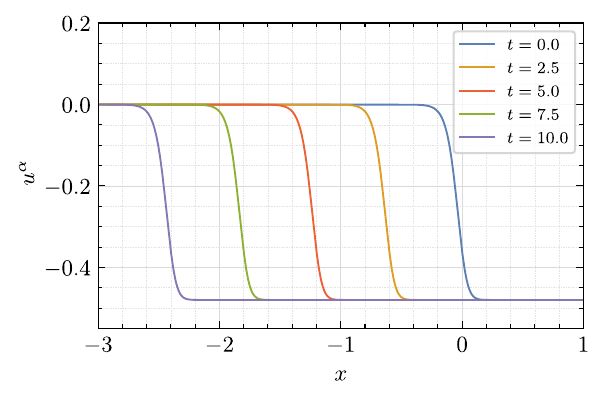}
    \end{subfigure}%
    \begin{subfigure}{0.5\textwidth}
        \centering
        \includegraphics[width=1\linewidth]{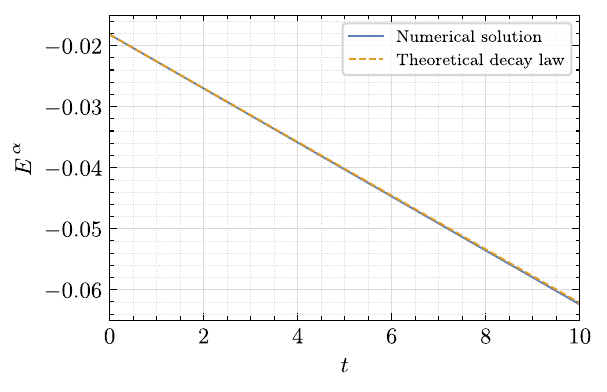}
    \end{subfigure}
    \caption{Left: The numerical solution to the AIVP corresgivenponding to \eqref{eq:IVP_sechtanh} at different times. Right: the evolution of the energy computed from the numerical solution, compared with the reference decay law, computed by solving the ODE in time. The parameters are $N=2000, \gamma=10^{-4}, \varepsilon=10^{-2}, \alpha=2\times10^3, \beta=10^{-6}$ and the final time is $T=10$. }
    \label{fig:sechtanh}
\end{figure}
\begin{figure}[h!]
        \begin{subfigure}{0.5\textwidth}
        \centering
        \includegraphics[width=\linewidth]{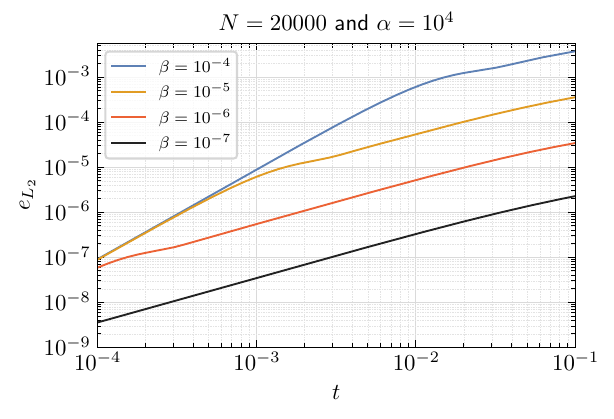}
    \end{subfigure}%
    \begin{subfigure}{0.5\textwidth}
        \centering
        \includegraphics[width=1\linewidth]{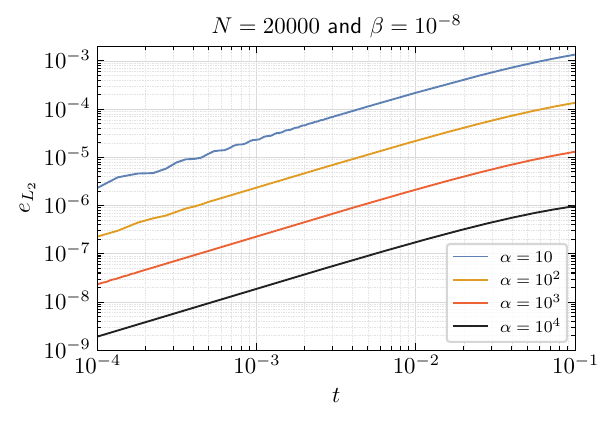}
    \end{subfigure}
\caption{Left: $L^2$-error as a function of time at fixed $\alpha=10^4$ with variable $\beta$. Right: $L^2$-error as a function of time at fixed $\beta$ with variable $\alpha$. Other parameters are $N=20000, \gamma=10^{-4}, \varepsilon=10^{-2}$ and $T=0.1$. } 
    \label{fig:sechtanh_relaxation}
\end{figure}
In Figure \ref{fig:sechtanh_relaxation}, we display errors for different pairings of the approximation parameters $\alpha$ and $\beta$. The error 
grows linearly in time as expected. For  fixed $\alpha$ or $\beta$ we observe a clear error decay if $\beta$ tends to zero or $\alpha$ increases, respectively. 
\begin{figure}[h!]
    \centering
    \begin{subfigure}{0.32\textwidth}
        \includegraphics[width=\textwidth]{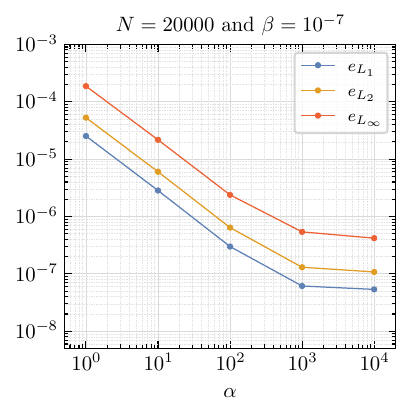}
    \end{subfigure}
        \begin{subfigure}{0.32\textwidth}
        \includegraphics[width=\textwidth]{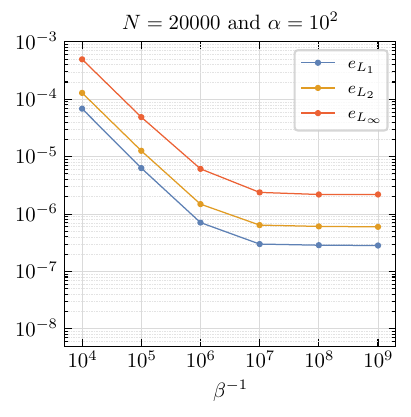}
    \end{subfigure}
        \begin{subfigure}{0.32\textwidth}
        \includegraphics[width=\textwidth]{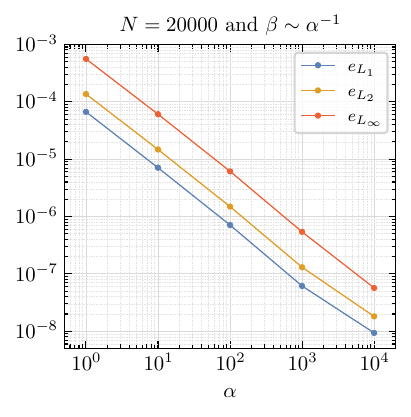}
    \end{subfigure}
    \caption{Convergence plots in different error norms. The choice of parameters is as in Figure \ref{fig:sechtanh_relaxation} above. Left: Error decay and saturation for for fixed $\beta$ and $\alpha \to \infty$. Middle:  Error decay and saturation for for fixed $\beta$ and $\alpha \to \infty$. Right: Linear error decay for $\beta = \alpha^{-1}$ and $\alpha \to \infty$.}
    \label{fig:placeholder}
\end{figure}

The plots in  Figure \ref{fig:placeholder} support our analytical findings from Theorem \ref{convergence_ratesdd}. As expected,  to decrease the error, both parameters $\alpha$ and $\beta$  have to be tuned. The linear error 
decay that can be observed from the right plot in Figure \ref{fig:placeholder} is  computed using the scaling $\beta = \gamma\alpha^{-1}$, which corresponds to the scaling $\beta\sim\alpha^{-1}$
from Theorem \ref{convergence_ratesdd}. In Table \ref{table:eoc}, we display corresponding numerical quantities from which we compute the 
experimental order of convergence  $O_{L^2}$.

\begin{table}[H]
\centering
\begin{tabular}{|ccc|ccc|ccc|}
\hline
\multicolumn{3}{|c|}{Fixed $\beta$, variable $\alpha$} &
\multicolumn{3}{c|}{Fixed $\alpha$, variable $\beta$} &
\multicolumn{3}{c|}{Scaled values $\beta\sim\alpha^{-1}$} \\
\hline
$\alpha$ & $e_{L^2}$ & $O_{L^2}$ &
$\beta$ & $e_{L^2}$ & $O_{L^2}$ &
$\alpha$ & $e_{L^2}$ & $O_{L^2}$ \\
\hline
$10^0$     & $5.29 \times 10^{-5}$  & --    & $10^{-5}$ & $1.27 \times 10^{-5}$ & --    & $10^0$     & $1.36 \times 10^{-4}$ & -- \\
$10^1$    & $6.03 \times 10^{-6}$  & 0.94  & $10^{-6}$ & $1.49 \times 10^{-6}$ & 0.93  & $10^1$    & $1.47 \times 10^{-5}$ & 0.97 \\
$10^2$   & $6.42 \times 10^{-7}$  & 0.97  & $10^{-7}$ & $6.42 \times 10^{-7}$ & 0.37  & $10^3$   & $1.49 \times 10^{-6}$ & 1.00 \\
$10^4$  & $1.31 \times 10^{-7}$  & 0.69  & $10^{-8}$ & $6.09 \times 10^{-7}$ & 0.02  & $10^3$  & $1.31 \times 10^{-7}$ & 1.06 \\
$10^4$ & $1.08 \times 10^{-7}$  & 0.08  & $10^{-9}$ & $5.98 \times 10^{-7}$ & 0.01  & $10^4$ & $1.82 \times 10^{-8}$ & 0.86 \\
\hline
\end{tabular}
\caption{Experimental order of convergence for the setting as in Figure \ref{fig:placeholder}. The right column clearly displays a first-order convergence whereas we could prove in Theorem \ref{theo:dd} only square-root behaviour.   \label{table:eoc}}
\end{table}

\subsubsection{Undercompressive shock  wave(modified KdV-Burgers)}
The next OIVP is given  for $\gamma >0 $ by 
\begin{equation}
\begin{cases}
u_t + u^2u_x  = \varepsilon u_{xx} - \gamma u_{xxx}, \\
u(x,0) = u_0(x) = \frac{1}{2}\prn{ u_+ + u_- - \abs{u_+-u_-} \tanh\prn{A x/\sqrt{\gamma}}},
    \end{cases}
    \label{eq:IVP_mkdV}
\end{equation}
where 
\begin{equation*}
u_- =  -\sqrt{\frac{2\varepsilon^2}{\gamma}}, \quad u_+ = -u_- -\, \sqrt{\frac{2\varepsilon^2}{9\gamma}}, \quad A =\frac{u_- - u_+}{2\sqrt{2}}. 
\end{equation*}
For the given values of the solution parameters, this corresponds to a traveling wave solution $u(x,t) = u_0(x-Vt)$, of velocity $V=\prn{u_+^3-u_-^3}/\prn{u_+ - u_-}$\cite{el2017dispersive} and consists in a transition from $u_-$ to $u_+$, occurring over a region of space of a length scaling with $\sqrt{\gamma}$. We consider for this test case a sharp transition by setting $\gamma=10^{-5}$. The computational domain is $I_c = [-0.1,2]$, discretized over $N =10000$ cells. We set $\varepsilon =10^{-2}$. Note that for these values of the parameters, we have approximately $u_- \simeq -4.47$, $u_+ \simeq 3.97$, $A \simeq -2.99 $ and  $V\simeq 18.02$. The numerical results are gathered in Figure~\ref{fig:mKdV_tanh}. Besides, we show here the influence of the relaxation parameters on the solution, and we consider values of $\alpha$ ranging from $100$ to $2500$ and values of $\beta$ from $10^{-6}$ down to $10^{-11}$ and for which the results are shown on Figure~\ref{fig:mKdV_tanh_rel}.  \\
We clearly observe that the first-order approximate system is able to display the undercompressive shock wave in the underlying cubic conservation 
law.  The numerical computation of undercompressive shock waves is difficult
because  its precise form depends sensitively on the ratio between diffusive and dispersive parameters, see \cite{el2017dispersive}. These can hardly be 
controlled in numerical schemes unless the shock layer is resolved by the mesh \cite{LR00}. Our new hyperbolic-parabolic approximation \eqref{hyperbolic_parabolic_system}  allows for stable computations even for underresolved wave structures.

\begin{figure}[h!]
    \centering
    \begin{subfigure}{0.5\textwidth}
    \includegraphics[width=\linewidth]{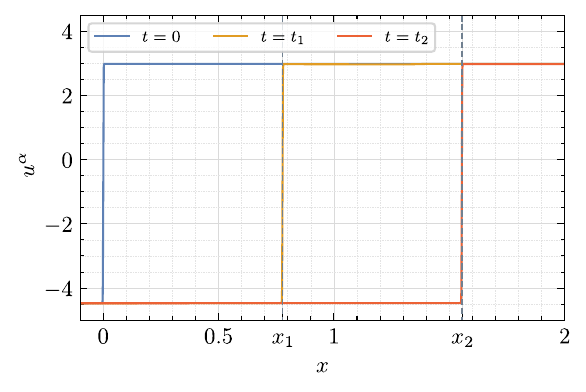}
    \end{subfigure}
    \begin{subfigure}{0.49\textwidth}
    \includegraphics[width=\linewidth]{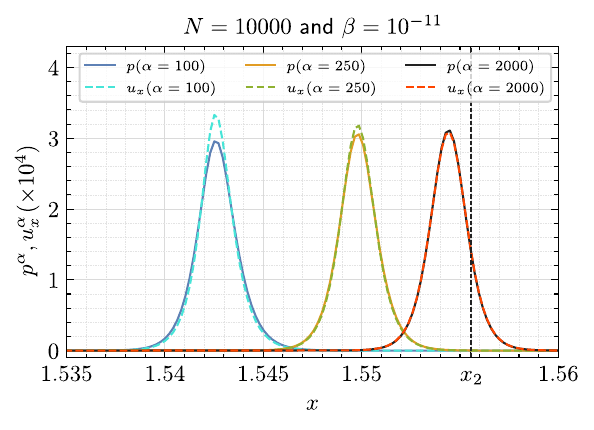}
    \end{subfigure}
    \caption{Left: Numerical solution for $u^\alpha$ at times $t_0 = 0$, $t_1 = 0.05$, and $t_2 = 0.1$, illustrating the evolution of the undercompressive shock. $x_{1,2}=Vt_{1,2}$ are the theoretical positions of the shock at $t=t_{1,2}$. Right: Comparison between $p^\alpha$ and the finite-difference approximation $u^\alpha_x$ at $t_2 = 0.1$ for several values of $\alpha$. 
}
    \label{fig:mKdV_tanh}
\end{figure}
\begin{figure}[h!] 
    \begin{subfigure}{0.5\textwidth}
        \centering
        \includegraphics[width=\linewidth]{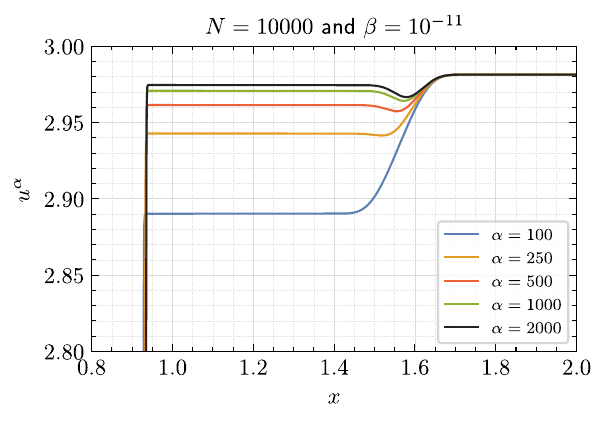}
    \end{subfigure}%
            \begin{subfigure}{0.5\textwidth}
        \centering
        \includegraphics[width=\linewidth]{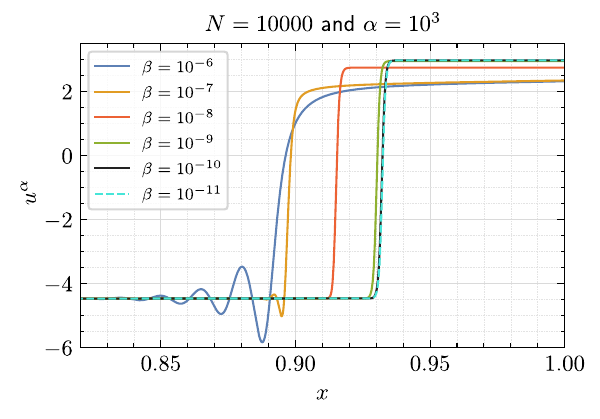}
    \end{subfigure}%

    \caption{Left:  Zoom at the trailing edge of the wave. For $\alpha \to \infty$ we observe convergence to the limit traveling wave. For smaller values of $\alpha$, another rarefaction wave gets visible. Note that the approximate system \eqref{hyperbolic_parabolic_system} may trigger up to four elementary waves.  Right: for fixed value of $\alpha$ and $\beta \to 0$, we observe oscillations and spurious wave speeds for larger values of  $\beta$ which are eliminated with $\beta$ vanishing.}
    \label{fig:mKdV_tanh_rel}
\end{figure}
The results show that different values of the relaxation parameters exhibit different solution structures, but ultimately show convergence towards the solution of the AIVP. Given that the approximate system is a $(4\times 4)$-system, the appearance of a rarefaction wave of vanishing amplitude as shown in Figure \ref{fig:mKdV_tanh_rel} (Right) in the limit $\alpha \to \infty$ is not surprising. 
On Figure \ref{fig:mKdV_tanh_rel} (Left), the numerical solution shows more drastic changes in its structure as well as the development of spurious oscillations of non-negligible amplitude for $\beta$ far from the stiff limit. This confirms the fact that both $\alpha$ and $\beta$ have to be scaled carefully in order to correctly capture the original solution's behavior. 
\subsubsection{Dark and bright solitons (Gardner equation)}
As final example, we consider as OIVP  a  Gardner equation  given by 
\begin{equation}
\begin{cases}
u_t + 6(u - k u^2)u_x + u_{xxx} = 0, \\
u(x,0) = u_{d,b}(x),
    \end{cases}
    \label{eq:IVP_gardnersol}
\end{equation}
where the initial conditions $u_d(x)$ and $u_b(x)$, corresponding to dark or bright solitons, respectively, write as \cite{kamchatnov2012undular}
\begin{align*}
& u_d(x) = u_3 - \frac{u_3 - u_2}{
\cosh^{2}\prn{\theta(x)} - \dfrac{u_3 - u_2}{u_3 - u_1}\sinh^{2}\prn{\theta(x)}}, \\
& u_b(x) = u_2 + \frac{u_3 - u_2}{
\cosh^{2}\prn{\theta(x)} - \dfrac{u_3 - u_2}{u_4 - u_2}\sinh^{2}\prn{\theta(x)}
}.
\end{align*}
Here $u_1\le u_2\le u_3 \le u_4$ and $\theta(x) = \frac{1}{2}\sqrt{k(u_3 - u_1)(u_4 - u_2)}x$. The exact solution to \eqref{eq:IVP_gardnersol} is then given by $u_{b,d}(x-V t)$ where the velocity $V$ is given as 
\begin{equation}
V = k \prn{u_1 u_2 + u_1 u_3 + u_1 u_4 + u_2 u_3 + u_2 u_4 + u_3 u_4}.
\end{equation}
These solutions are obtained from a broader class of solutions to the Gardner equations in the limits $u_3\to u_4$ for $u_d$ and $u_2\to u_1$ for $u_b$, see \cite{kamchatnov2012undular}. For the numerical simulation of the associated AIVP, we consider a small parameter $\epsilon>0$, and we take for the dark soliton solution 
\begin{equation*}
    k=1, \quad u_1 = 0, \quad u_2 = \epsilon, \quad u_3 = u_4 = 1-\epsilon,
\end{equation*}
and analogously for the bright soliton 
\begin{equation*}
    k=1, \quad u_1 = u_2 = \epsilon, \quad u_3 = 1-\epsilon, \quad u_4 = 1.
\end{equation*}
This results in table-top-like solitons for $\epsilon \ll1$. For both parameterizations, we take $\epsilon=10^{-4}$ and the value of the traveling wave velocity is $V=1-\epsilon^2\simeq 1$. 
The rest of the parameters for both simulations are as follows. The computational domain $I_c = [-50,50]$ is  discretized over $N=2000$ cells. We then choose $\alpha=10^3$ and $\beta=10^{-6}$. Periodic boundary conditions are used, and we set the final time to $T=500/V$, corresponding to $10$ periods of the solution over the computational domain before returning to its initial state. The numerical results are reported in Figure \ref{fig:gardner_solitons} and underline the validity of our overall approach to approximate diffusive-dispersive equations by extended hyperbolic-parabolic systems.\\

\begin{figure}[h!]
    \centering
    \includegraphics[width=0.49\linewidth]{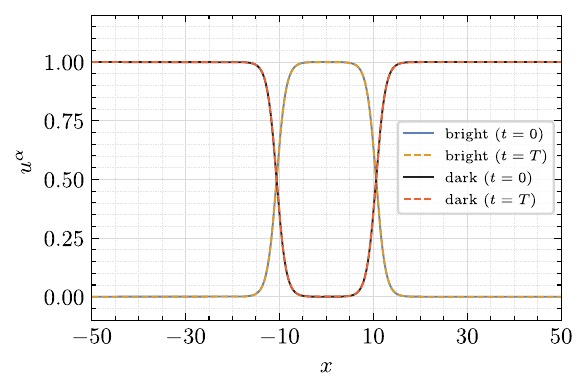}
    \includegraphics[width=0.49\linewidth]{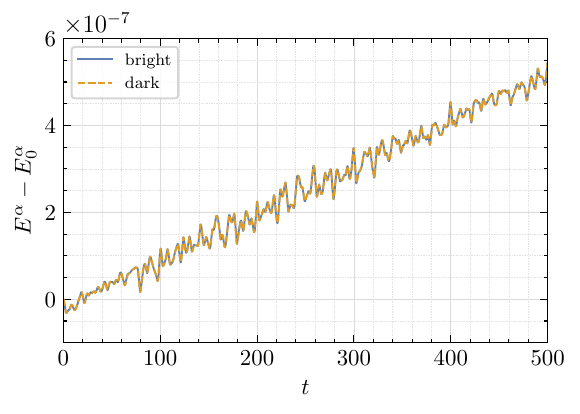}
    \caption{Left: Numerical solutions to the AIVP corresponding to \eqref{eq:IVP_gardnersol} for both bright and dark solitons. Right: Evolution of the total energy error $E\rel-E\rel_0$ over time. $I_c = [-50,50], N=2000, \alpha=10^3, \beta=10^{-6}$ and $T=500/V$.  }
    \label{fig:gardner_solitons}
\end{figure}
\section{Conclusions and future outlook}\label{sec: conclusions}
We have introduced a novel framework for the hyperbolic and hyperbolic-parabolic approximation for a broad class of nonlinear dispersive and diffusive-dispersive equations. 
The hyperbolic approximation \eqref{hyperbolic_system} is constructed such that a complete set of Riemann invariants can be determined. It is shown to have a Hamiltonian structure that is consistent with the limit dispersive equations.
Enforcing the monotonicity on the flux derivative,
we have found that the hyperbolic approximate system \eqref{hyperbolic_system} is equipped with an entropy/entropy flux pair that resembles the energy of the limit dispersive equations.
Moreover, solutions of the hyperbolic–parabolic approximation \eqref{hyperbolic_parabolic_system}
dissipate this entropy. Based on these findings, we have employed the relative entropy framework to prove the convergence of solutions of the approximate systems towards smooth solutions of the dispersive and diffusive-dispersive limit equations.\\
Due to the likewise standard hyperbolic and hyperbolic structure of the approximate systems, it becomes possible to use high-resolution methods from the realm of hyperbolic balance laws for the sake of numerical simulation. 
For a wide range of diffusive-dispersive equations, we display the effectiveness, accuracy, and robustness of a semi-implicit second-order two-step Lax-Wendroff method. In particular, our focus in the numerical examples was to capture the dynamics of non-classical shock waves, including dispersive and undercompressive shock waves, which are very rare to find in the literature. The numerics suggest convergence of our approximate systems for a significantly larger class of fluxes, independent of the nonlinearity of the flux.   
The study refers to dispersive equations with nonconvex energy, such as the KdV equation. 


As mentioned, the proposed first-order system \eqref{hyperbolic_system}  possesses a full set of Riemann invariants. This structure opens the door to possible global existence results for $L^\infty$-initial data via compensated compactness, as well as the development of numerical schemes based on exact Riemann solvers. It would be interesting to analyze whether one can obtain the entropy solution of the hyperbolic system \eqref{hyperbolic_system} from the vanishing diffusion (viscosity) limit of \eqref{hyperbolic_parabolic_system}, i.e., $\varepsilon\rightarrow 0$. In this work, we have chosen the scaling between diffusion and dispersion to be $1$. For the alternative scaling $\varepsilon^2$, one can also study the diffusion–dispersion limit of the proposed systems as in \cite{corli2012singular, corli2014parabolic}. We leave these issues for future investigation.\\
Furthermore, this approach of relaxing higher-order PDEs via energy-consistent lower-order approximations can also be applied to more complex higher-order PDEs, having, for example, nonlinearities in higher-order terms \cite{barthwal_thinfilm,dhaouadi2025first}. Moreover, this relaxation approach can be readily adapted to coupled multi-dimensional systems such as Navier-Stokes-Korteweg-type systems, where the energy structure of the underlying mechanisms can be crucial in designing accurate and stable approximate systems \cite{dhaouadi2022first,HRYZ,KMR}.\\\\
\textbf{Acknowledgements}~ Funding by the Deutsche Forschungsgemeinschaft (DFG, German Research Foundation) - SPP 2410 Hyperbolic Balance Laws in Fluid Mechanics: Complexity, Scales, Randomness (CoScaRa) is gratefully acknowledged.\\
F.D. acknowledges the support from NextGenerationEU, Azione 247 MUR
Young Researchers – SoE line.\\\\
\textbf{CRediT authorship contribution statement}~\\
{\bf{Rahul Barthwal:}} Modeling, mathematical analysis, investigation, methodology, writing— original draft, review \& editing, {\bf{Firas Dhaouadi:}} modeling, methodology,  numerical simulations, investigation, validation, writing -original draft, review \& editing, {\bf{Christian Rohde:}} investigation, writing-original draft, review \& editing, supervision. 
\bibliographystyle{mystyle}
\bibliography{references}

\end{document}